\documentclass{amsart}
\usepackage{graphicx}
\usepackage{amssymb}
\usepackage{amsfonts}
\swapnumbers
\newtheorem{theorem}{Theorem}[section]
\newtheorem{lemma}[theorem]{Lemma}
\newtheorem{corollary}[theorem]{Corollary}
\newtheorem{proposition}[theorem]{Proposition}
\theoremstyle{definition}

\newtheorem{remark}[theorem]{Remark}

\numberwithin{equation}{section}
 \theoremstyle{plain}
 
 \numberwithin{equation}{section} 
 \numberwithin{figure}{section} 
 \theoremstyle{plain}
 \theoremstyle{remark}
 \newtheorem*{acknowledgement*}{Acknowledgement}

\newcommand{\cB}{{\mathcal B}}
\newcommand{\cC}{{\mathcal C}}

\newcommand{\cF}{{\mathcal F}}
\newcommand{\cG}{{\mathcal G}}
\newcommand{\cH}{{\mathcal H}}
\newcommand{\cI}{{\mathcal I}}
\newcommand{\cJ}{{\mathcal J}}
\newcommand{\cK}{{\mathcal K}}
\newcommand{\cL}{{\mathcal L}}

\newcommand{\cQ}{{\mathcal Q}}
\newcommand{\cR}{{\mathcal R}}
\newcommand{\cS}{{\mathcal S}}

\newcommand{\cW}{{\mathcal W}}
\newcommand{\cX}{{\mathcal X}}
\newcommand{\cY}{{\mathcal Y}}

\newcommand{\Te}{{\Theta}}
\newcommand{\te}{{\theta}}
\newcommand{\vt}{{\vartheta}}
\newcommand{\Om}{{\Omega}}
\newcommand{\om}{{\omega}}
\newcommand{\ve}{{\varepsilon}}
\newcommand{\del}{{\delta}}
\newcommand{\Del}{{\Delta}}
\newcommand{\gam}{{\gamma}}
\newcommand{\Gam}{{\Gamma}}
\newcommand{\vf}{{\varphi}}
\newcommand{\vr}{{\varrho}}
\newcommand{\Sig}{{\Sigma}}
\newcommand{\sig}{{\sigma}}
\newcommand{\al}{{\alpha}}

\newcommand{\ka}{{\kappa}}
\newcommand{\la}{{\lambda}}
\newcommand{\La}{{\Lambda}}

\newcommand{\vs}{{\varsigma}}

\newcommand{\bbR}{{\mathbb R}}

\newcommand{\bbZ}{{\mathbb Z}}
\newcommand{\bbI}{{\mathbb I}}

\newcommand{\bfW}{{\bf W}}

\newcommand{\fG}{{\mathfrak G}}

\begin{document}
\title[]{Some strong limit theorems in averaging}%
 \vskip 0.1cm
 \author{ Yuri Kifer\\
\vskip 0.1cm
 Institute  of Mathematics\\
Hebrew University\\
Jerusalem, Israel}%
\address{
Institute of Mathematics, The Hebrew University, Jerusalem 91904, Israel}
\email{ kifer@math.huji.ac.il}%

\thanks{ }
\dedicatory{  }
 \date{\today}
\begin{abstract}\noindent
The paper deals with the fast-slow motions setups in the discrete time
$X^\ve((n+1)\ve)=X^\ve(n\ve)+\ve B(X^\ve(n\ve),\xi(n))$,  $n=0,1,...,[T/\ve]$
and the continuous time $\frac {dX^\ve(t)}{dt}=B(X^\ve(t),\xi(t/\ve)),\, t\in [0,T]$
 where $B$ is a smooth in the first variable vector function and $\xi$ is a sufficiently fast mixing stationary
 stochastic process. It is known since \cite{Kha66} that if $\bar X$ is the averaged motion
 then $G^\ve=\ve^{-1/2}(X^\ve-\bar X)$ weakly converges to a Gaussian process $G$. We will
 show that for each $\ve$ the processes $\xi$ and $G$ can be redefined on a sufficiently
 rich probability space without changing their distributions so that
 $E\sup_{0\leq t\leq T}|G^\ve(t)-G(t)|^{2M}
 =O(\ve^{\del}),\,\del>0$ which gives also $O(\ve^{\del/3})$ Prokhorov distance estimate between the
 distributions of $G^\ve$ and $G$. This provides also convergence estimates in the
 Kantorovich--Rubinstein (or Wasserstein) metrics. In the product case $B(x,\xi)=\Sig(x)\xi$
 we obtain also almost sure convergence estimates of the form $\sup_{0\leq t\leq T}|G^\ve(t)-G(t)|
 =O(\ve^\del)$ a.s., as well as the Strassen's form of the law of iterated logarithm for
 $G^\ve$. We note that our mixing assumptions are adapted to fast
 motions generated by important classes of dynamical systems.

\end{abstract}
\maketitle
\markboth{Yu.Kifer}{Limit theorems in averaging}
\renewcommand{\theequation}{\arabic{section}.\arabic{equation}}
\pagenumbering{arabic}

\section{Introduction}\label{sec1}\setcounter{equation}{0}
Let $X^\ve(t)=X_x^\ve(t)$ be the solution of a system of ordinary differential equations having
 the form
\begin{equation}\label{1.1}
\frac {dX^\ve(t)}{dt}=B(X^\ve(t),\xi(t/\ve)),\,\, X^\ve(0)=x,\, t\in[0,T]
\end{equation}
where $B(x,\xi(s))$ is a (random) bounded twice differentiable in the first variable
vector field on $\bbR^d$ and $\xi$ is a stationary process on a Polish space $\cY$
which is viewed as a fast motion while $X^\ve$ is considered as a slow motion. Under the ergodicity
assumption the limit
\begin{equation}\label{1.2}
\bar B(x)=\lim_{t\to\infty}\frac 1t\int_0^tB(x,\xi(s))ds
\end{equation}
exists almost surely (a.s.) and it is Lipschitz continuous, as well. Hence, the solution
$\bar X(t)=\bar X_x(t)$ of the equation
\begin{equation}\label{1.3}
\frac {d\bar X(t)}{dt}=\bar B(\bar X(t)),\,\, \bar X(0)=x,\, t\in[0,T]
\end{equation}
exists and it is called the averaged motion. It is well known that $X_x^\ve(t)$ and $\bar X_x(t)$
are close uniformly in time on bounded time intervals (cf. \cite{Kha66}). This is a version of the
averaging principle which was used in celestial mechanics already in the 18th century though
its rigorous justification was obtained only in the middle of the 20th century (see\cite{BM}).

The next natural step was to study the error $X_x^\ve(t)-\bar X_x(t)$ of the averaging approximation
and it was shown in \cite{Kha66} that the normalized difference $\ve^{-1/2}(X_x^\ve(t)-\bar X_x(t))$
converges weakly as $\ve\to 0$ to a $d$-dimensional Gaussian Markov process $G(t)$ which solves
the linear stochastic differential equation
\begin{equation}\label{1.4}
dG(t)=\nabla\bar B(\bar X_x(t))G(t)dt+\sig(\bar X_x(t))dW(t),\, G(0)=0
\end{equation}
where $W(t)$ is a standard $d$-dimensional Brownian motion, the diffusion matrix $\sig$ is obtained
via certain limit and for each vector $b(x)=(b_1(x),...,b_d(x))$ we denote by $\nabla b(x)$ the matrix
whose $(i,j)$-th element is $\partial b_i(x)/\partial x_j$. In \cite{Ki03} a stronger convergence result
 was obtained considering in addition to the process $G(t)$ another process $H(t)=H^\ve(t)$ solving the stochastic
 differential equation
 \begin{equation}\label{1.5}
 dH^\ve(t)=\bar B(H^\ve(t))dt+\sqrt\ve\sig(H^\ve(t))dW(t),\, H^\ve_x(0)=H^\ve(0)=x
 \end{equation}
 which was suggested by K.Hasselmann (2021 Nobel prize in physics) in the study of climate
 evolution (see \cite{Has}). In some models of weather--climate interactions the former is
 viewed as a fast chaotic and the latter as the slow motions in the averaging setup above. Supported by
 heuristic arguments the process $H^\ve$ was introduced in \cite{Has} as an approximation
 for the slow motion $X^\ve$. It was shown in \cite{Ki03} that under sufficiently fast mixing
 assumptions the process $\xi(t)$ can be redefined on a larger probability space preserving
its distribution where there exists a standard Brownian motion $W$ such that for $G(t)$ and
$H^\ve(t)$ constructed by (\ref{1.4}) and (\ref{1.5}) with such $W$ the $L^2$-norms
$E\sup_{0\leq t\leq T}|X^\ve(t)-\bar X(t)-\sqrt\ve G(t)|^2$ and $E\sup_{0\leq t\leq T}|X^\ve(t)
-H^\ve(t)|^2$ have the order $\ve^{1+\del}$ with $\del>0$ which provides also estimates
for the Prokhorov distance between distributions
of pairs $X^\ve-\bar X,\,\sqrt\ve G$ and $X^\ve,\, H^\ve$. It is important to observe that both
 \cite{Kha66} and \cite{Ki03} assume that the process $\xi(t)$ is sufficiently fast mixing
with respect to $\sig$-algebras generated by itself which greatly restricts applications
to the case when $\xi(t)$ is generated by a dynamical system as will be discussed later on.
Similar results when $\xi(t)$ is a diffusion depending on the slow motion $X^\ve$ (the fully
coupled case) were obtained in \cite{BK}. Still, the statistical study of the slow motion driven
by a deterministic fast motion fits more the spirit of the chaos theory, and so we will deal in
this paper with fast motions $\xi$ which can be generated by important classes of dynamical systems.

In this paper we start with the discrete time setup given by the recurrence relation
\begin{equation}\label{1.6}
X^\ve((n+1)\ve)=X^\ve(n\ve)+\ve B(X^\ve(n\ve),\xi(n))
\end{equation}
where $B(x,\xi)$ is a smooth vector field in $x$ on $\bbR^d$ Borel measurably dependent on $\xi$
and $\xi(n)$ is a vector valued
stationary process while sufficiently fast mixing is assumed with respect to a two parameter
family of $\sig$-algebras $\cF_{mn}$ as described in the next section. Each random vector $\xi(n)$
is not supposed to be $\cF_{nn}$-measurable and, instead, we assume that it is well approximated
by its conditional expectations with respect to $\sig$-algebras $\cF_{n-m,n+m}$. As will be explained 
in more details in the next section, this will enable us
to consider processes generated by a large class of dynamical systems, i.e. when $\xi(\om,n)=f(F^n\om)$
 for a probability preserving transformation $F$ and a H\" older continuous vector function $f$. 
 The relevant dynamical systems include hyperbolic ones which is the important family of chaotic systems
  considered  sometimes as a model for weather evolution.

For $t\in[0,T]$ we set
$X^\ve(t)=X^\ve([t/\ve]\ve)$ and show that the process $\xi(n)$ can be redefined on a richer
probability space preserving its distribution where there exists a Brownian motion $W=W_\ve$
 such that for $G=G_\ve$ solving (\ref{1.4}) with such $W=W_\ve$ we have the moment estimates of the form
 $E\sup_{0\leq t\leq T}|\ve^{-1/2}(X^\ve(t)-\bar X(t))-G_\ve(t)|^{2M}=O(\ve^{\del})$, where we can take
 $\del=(270d)^{-1}$, which gives also the Prokhorov and the Kantorovich--Rubinstein (or Wasserstein)
 distance estimate of $\ve^{\del/3}$ between distributions of $\ve^{-1/2}(X^\ve(t)-\bar X(t))$ and $G$.
  These results are easily extended to the continuous time case (\ref{1.1}) if the corresponding
  process $\xi(t)$ is sufficiently fast
 mixing. But in applications to dynamical systems the latter is quite restrictive and we consider
 a more realistic situation where the process $\xi(t)$ is obtained via the so called suspension
 construction over a sufficiently fast mixing discrete time probability preserving transformation
 which requires some work and will be done by reducing the problem to the discrete time case.
 The weak limit theorem and moderate deviations results in this setup were obtained in \cite{Ki95}
 and similar weak convergence results for some additional classes of dynamical systems were derived
 in \cite{Pe} under somewhat different assumptions. Of course, weak convergence results cannot provide
 any speed estimates while the results of this paper yield explicit moment and Prokhorov distance estimates
  for each $\ve>0$. In the product case $B(x,\xi)=\Sig(x)\xi$
 we obtain also almost sure convergence estimates of the form $\sup_{0\leq t\leq T}|\ve^{-1/2}(X^\ve(t)-\bar X(t))-G_\ve(t)|
 =O(\ve^\del)$ a.s., as well as the Strassen form of the functional law of iterated logarithm saying that with
 probability one as $\ve\to 0$ the set of limit points of random functions $(2\ve\log\log(1/\ve))^{-1/2}(X^\ve_x-\bar X_x)$
 coincides with certain compact set in the space of continuous functions on $[0,T]$.
 
 At the end of the introduction we will mention another but quite different averaging problem started from \cite{Kha66+}
 and attracted recently the renewed interest (see \cite{KM}, \cite{Ki22}, \cite{Ki24}, \cite{FK} and references there). 
 The main difference of the setup considered there with the present paper is that here we are interested in the normalized
 deviation of $X^\ve$ from the averaged motion $\bar X$ for the time of order $1/\ve$ while in the other setup the averaging
 motion is zero and we are waiting for the time of order $1/\ve^2$ to see what happens with  $X^\ve$. In the first case 
 considered here we end up with a Gaussian process (or a close to it Hasselmann's diffusion with a small parameter). In the
 second setup the limiting process there turns out to be a quite general diffusion.  Observe also that unlike the present paper,
  it was possible to obtain strong moment convergence estimates in the
  second setup only under very special assumptions on coefficients of \cite{Ki22} while without it only weak or almost sure
  results were obtained and they relied on the rough paths theory technique. The approaches to these two types of averaging
   problems are mostly different but when dealing with strong limit theorems in \cite{Ki22}, \cite{Ki24}, \cite{FK} and in 
   the present paper the methods have a non empty intersaction  around preparations to apply the strong approximation theorem
    from \cite{BP}.

  The structure of this paper is the following. In the next section we provide necessary definitions and give
   precise statements of our results. Section \ref{sec3} is devoted to necessary estimates both of general nature
   and more specific to our problem, as well as characteristic functions approximations needed in Section \ref{sec4}
    for the strong approximation theorem. In Section \ref{sec5} we deal with the continuous time case relying on
    Sections \ref{sec3} and \ref{sec4} after certain discretization. In Section 6 we obtain the a.s. approximation results
    and the functional law of iterated logarithm in the product case.

\section{Preliminaries and main results}\label{sec2}\setcounter{equation}{0}
\subsection{Discrete time case}
We start with the discrete time setup which consists of a complete probability space
$(\Om,\cF,P)$, a stationary sequence of $d$-dimensional random vectors $\xi(n)=(\xi_1(n),...,\xi_d(n))$,
 $-\infty<n<\infty$ and
a two parameter family of countably generated $\sig$-algebras
$\cF_{m,n}\subset\cF,\,-\infty\leq m\leq n\leq\infty$ such that
$\cF_{mn}\subset\cF_{m'n'}\subset\cF$ if $m'\leq m\leq n
\leq n'$ where $\cF_{m\infty}=\cup_{n:\, n\geq m}\cF_{mn}$ and $\cF_{-\infty n}=\cup_{m:\, m\leq n}\cF_{mn}$.
It is often convenient to measure the dependence between two sub
$\sig$-algebras $\cG,\cH\subset\cF$ via the quantities
\begin{equation}\label{2.1}
\varpi_{b,a}(\cG,\cH)=\sup\{\| E(g|\cG)-Eg\|_a:\, g\,\,\mbox{is}\,\,
\cH-\mbox{measurable and}\,\,\| g\|_b\leq 1\},
\end{equation}
where the supremum is taken over real functions and $\|\cdot\|_c$ is the
$L^c(\Om,\cF,P)$-norm. Then more familiar $\al,\rho,\phi$ and $\psi$-mixing
(dependence) coefficients can be expressed via the formulas (see \cite{Bra},
Ch. 4 ),
\begin{eqnarray*}
&\al(\cG,\cH)=\frac 14\varpi_{\infty,1}(\cG,\cH),\,\, \rho(n)=\varpi_{2,2}(n),\,\phi(\cG,\cH)=\frac 12\varpi_{\infty,\infty}(\cG,\cH)\\
&\mbox{and}\,\,\,\psi(\cG,\cH)=\varpi_{1,\infty}(\cG,\cH).
\end{eqnarray*}
We set also
\begin{equation*}
\varpi_{b,a}(n)=\sup_{k\geq 0}\varpi_{b,a}(\cF_{-\infty,k},\cF_{k+n,\infty})
\end{equation*}
and accordingly
\[
\al(n)=\frac{1}{4}\varpi_{\infty,1}(n),\,\rho(n)=\varpi_{2,2}(n),\,
\phi(n)=\frac 12\varpi_{\infty,\infty}(n),\, \psi(n)=\varpi_{1,\infty}(n).
\]
Furthermore, by the real version of the Riesz--Thorin interpolation
theorem or the Riesz convexity theorem (see \cite{Ga}, Section 9.3
and \cite{DS}, Section VI.10.11) whenever $\theta\in[0,1],\, 1\leq
a_0,a_1,b_0,b_1\leq\infty$ and
\[
\frac 1a=\frac {1-\theta}{a_0}+\frac \theta{a_1},\,\,\frac 1b=\frac{1-\theta}{b_0}+\frac \theta{b_1}
\]
then
\begin{equation*}
\varpi_{b,a}(n)\le 2(\varpi_{b_0,a_0}(n))^{1-\theta}(\varpi_{b_1,a_1}(n))^\theta.
\end{equation*}
In particular,  using the obvious bound $\varpi_{b_1,a_1}(n)\leq 2$
valid for any $b_1\geq a_1$ we obtain from (\ref{2.9}) for pairs
$(\infty,1)$, $(2,2)$ and $(\infty,\infty)$ that for all $b\geq a\geq 1$,
\begin{eqnarray}\label{2.2}
&\varpi_{b,a}(n)\le 4(2\alpha(n))^{\frac{1}{a}-\frac{1}{b}},\,
\varpi_{b,a}(n)\le 2^{1+\frac 1a-\frac 1b}(\rho(n))^{1-\frac 1a+\frac 1b}\\
&\mbox{and}\,\,\varpi_{b,a}(n)\le 2^{1+\frac 1a}(\phi(n))^{1-\frac 1a}.\nonumber
\end{eqnarray}
This enables us to replace our assumption below on the decay of the dependence coefficient $\varpi$ by the
corresponding assumptions on the more familiar dependence coefficients mentioned above.

We will assume that $B:\,\bbR^d\times\cY\to\bbR^d$ is a Borel map with a uniform bound on its $C^2$ norm in
the first argument
\begin{equation}\label{2.3}
\sup_{x\in\bbR^d}\sup_{y\in\cY}\max_{0\leq i,j,k\leq d}\max(|B_i(x,y)|,\,|\frac {\partial B_i(x,y)}{\partial x_j}|,\,\\
|\frac {\partial^2 B_i(x,y)}{\partial x_j\partial x_k}|)=L<\infty.
\end{equation}
Unlike \cite{Ki03}, in order to ensure more applicability of our results to dynamical systems,
we do not assume that $\xi(n)$ is $\cF_{nn}$-measurable and instead we will work with the moment approximation coefficient
\begin{eqnarray}\label{2.4}
& \rho(a,n)=\sup_{x,m}\max_{i,j,k}\max\big(\|B_i(x,\xi(m))\\
&-E(B_i(x,\xi(m))|\cF_{m-n,m+n})\|_a,\,\|\frac {\partial B_i(x,\xi(m))}{\partial x_j}-E(\frac {\partial B_i(x,\xi(m))}{\partial x_j}|\cF_{m-n,m+n})\|_a,\nonumber\\
&\|\frac {\partial^2 B_i(x,\xi(m))}{\partial x_j\partial x_k}-E(\frac {\partial^2 B_i(x,\xi(m))}{\partial x_j\partial x_k}|\cF_{m-n,m+n})\|_a\big).\nonumber
\end{eqnarray}
Observe that if we assume that $\cY$ is a Banach space with a norm $|\cdot |$, the functions
$B_i(x,y),\,\frac {\partial B_i(x,y)}{\partial x_j},\,\frac {\partial^2 B_i(x,y)}
{\partial x_j\partial x_k}$ are Lipschitz continuous in $y$ with a constant $L$ and
\begin{equation}\label{2.5}
\sup_m\|\xi(m)-E(\xi(m)|\cF_{m-n,m+n})\|_a\leq\frac 12L^{-1}\rho(a,n)
\end{equation}
then (\ref{2.4}) holds true. In particular, this is satisfied if $B(x,\xi(m))=\sig(x)\xi(m)$
where $\xi(m)$'s are random vectors and $\sig(x)$ is a matrix function with its $C^2$-norm
bounded by $L$.
To save notations we will still write $\cF_{mn}$, $\varpi_{b,a}(n)$ and $\rho(a,n)$ for $\cF_{[m][n]}$, $\varpi_{b,a}([n])$ and $\rho(a,[n])$, respectively, if 
$m$ and $n$ are not integers (or $\pm\infty$), where $[\cdot]$ denotes the integral part. We will assume that the coefficients $\varpi$ and $\rho$ decay
fast enough, namely that for some $K,M\geq 1$ large enough,
\begin{equation}\label{2.6}
D=\sum_{n=0}^\infty n^5(\varpi_{K,4M}(n)+\rho(K,n))<\infty.
\end{equation}
It turns out that the same proofs work in a seemingly more general setup when we hide the process $\xi(m)$ and consider instead an invertible probability preserving transformation $\vt:\,\Om\to\Om$,
so that $B(x,\xi(m))$ is replaced by $B(x,\vt^m\om)$ where $B:\,\bbR^d\times\Om\to\bbR^d$ is
measurable and satisfies the regularity conditions (\ref{2.3}) in the first variable. In fact, if
$\Om$ is a Lebesgue (standard probability) space it is always possible to pass from this second representation to the first one, so that this becomes just the matter of notations.

Define also $\hat B(x,\xi(m))=B(x,\xi(m))-EB(x,\xi(m))$,
\[
a_{ij}(x,y,m,n)=E(\hat B_i(x,\xi(m))\hat B_j(y,\xi(n)))\quad\mbox{and}\quad a_{ij}(x,m,n)
=a_{ij}(x,x,m,n)
\]
where $\hat B=(\hat B_1,...,\hat B_d)$ and $\hat B_i=B_i-EB_i$ with $B=(B_1,...,B_d)$.
It will be shown in the next section under the conditions of our assertions
below that for $i,j=1,...,d$ the limits
\begin{equation}\label{2.7}
a_{ij}(x)=\lim_{n\to\infty}\frac 1n\sum_{k=m}^{m+n}\sum_{l=m}^{m+n}a_{ij}(x,k,l)
=\lim_{n\to\infty}\frac 1n\sum_{k=0}^n\sum_{l=0}^na_{ij}(x,k,l)
\end{equation}
exist. We will see that under our conditions the matrix $A(x)=(a_{jk}(x))$ is symmetric and twice differentiable in $x$, and so it has a symmetric Lipschitz continuous in $x$ square root $\sig(x)$, i.e. we have the representation (see \cite{Fre} and Sections 5.2 and 5.3 in \cite{SV}),
\begin{equation}\label{2.8}
A(x)=\sig^2(x),
\end{equation}
and both the uniform bound of the norm and the Lipschitz constant of $\sig$ will be denoted again by $L$. In fact, for our purposes it suffices to have the representation $A(x)=\sig(x)\sig^*(x)$ with a Lipschitz continuous matrix $\sig$ where $\sig^*$ is the conjugate to $\sig$. Thus, there exists a
unique solution $H^\ve$ of the stochastic differential equation (\ref{1.5}). Set also $\hat G^\ve(t)
=\bar X(t)+\sqrt\ve G(t),\, \hat G^\ve(0)=\hat G^\ve_x(0)=x$ where $\bar X$ and $G$ are given
by (\ref{1.3}) and (\ref{1.4}),respectively.
Our main results in the discrete time case are the following.
\begin{theorem}\label{thm2.1}
Suppose that the conditions (\ref{2.3}) and (\ref{2.6}) hold true and a symmetric Lipschitz continuous
matrix $\sig(x)$ satisfying (\ref{2.8}) is fixed. Then for each $\ve>0$ the stationary process
$\xi(n),\, 0\leq n<\infty$ can be redefined preserving its distributions
 on a richer probability space where there exists a
standard Brownian motion $W=W_\ve$ so that the slow motion $X_x^\ve$, the diffusion $H_x^\ve$ and the Gaussian process $\hat G_x^\ve$ constructed with these newly defined processes and having the same initial condition $X^\ve(0)=H^\ve(0)=\hat G^\ve(0)=x$, satisfy
\begin{equation}\label{2.9}
E\sup_{0\leq t\leq T}|X^\ve_x(t)-H_x^\ve(t)|^{2M}\leq C_0(M)\ve^{M+\del}\quad\mbox{and}
\end{equation}
\begin{equation}\label{2.10}
E\sup_{0\leq t\leq T}|X^\ve_x(t)-\hat G_x^\ve(t)|^{2M}\leq C_0(M)\ve^{M+\del},
\end{equation}
for any integer $M\geq 1$, where we can take $\del=\frac 1{270d}$ and $C_0(M)>0$ does not depend on $\ve$ and
can be explicitly estimated from the proof. In particular,
\begin{equation}\label{2.11}
E\sup_{0\leq t\leq T}|\ve^{-1/2}(X^\ve(t)-\bar X_x(t))-G(t)|^{2M}\leq C_0(M)\ve^{\del}.
\end{equation}
Furthermore,
\begin{equation}\label{2.12}
E\sup_{0\leq t\leq T}|\ve^{-1/2}(H_x^\ve(t)-\bar X_x(t))-G(t)|^{2M}\leq 2^{2M-1}\hat C_0(M)\ve^{M},
\end{equation}
where $\hat C_0(M)>0$ does not depend on $\ve$, provided $G$ and $H_x$ are given by (\ref{1.4}) and (\ref{1.5}), respectively,
 with the same Brownian motion $W$.
\end{theorem}

While (\ref{2.12}) follows directly from Lemmas \ref{lem3.6} and \ref{lem3.7} from Section \ref{sec3}, in order to
obtain (\ref{2.9}) we will rely on the strong approximation theorem from Section \ref{sec4}. Then (\ref{2.10}) will
follow from (\ref{2.9}) and (\ref{2.12}) with an appropriate $C_0(M)>0$ which does not depend on $\ve$.

Recall, that the Prokhorov distance $\pi$ between two probability measures $\mu$ and $\nu$ on a
metric space $\cX$ with a distance function $d$ is defined by
\begin{eqnarray*}
&\pi(\mu,\nu)=\inf\{\ka>0:\,\mu(U)\leq\nu(U^\ka)+\ka\,\,\mbox{and}\\
&\nu(U)\leq\mu(U^\ka)+\ka\,\,\mbox{for any Borel set $U$ on $\cX$}\}
\end{eqnarray*}
where $U^\ka=\{ x\in\cX:\, d(x,y)<\ka\,\,$ for some $y\in U\subset\cX\}$ is the $\ka$-neighborhood
of $U$.  Recall also that the $L^q$ Wasserstein (or Kantorovich--Rubinstein) distance between
 two probability measures $\mu$ and $\nu$ on $\cX$ is defined by
 \[
 w_q(\mu,\nu)=\inf\{ (Ed^q(Q,R))^{1/q}:\,\cL(Q)=\mu\,\,\mbox{and}\,\, \cL(R)=\nu\}
 \]
 where the infimum is taken over all random points (variables) $Q$ and $R$ in $\cX$ with their distributions
  $\cL(Q)$ and $\cL(R)$ equal $\mu$ and $\nu$ respectively. From Theorem \ref{thm2.1} we obtain immediately that
 \begin{eqnarray}\label{2.13}
 &w_{2M}(\cL(\frac {X^\ve_x-\bar X_x}{\sqrt\ve}),\,\cL(G))\leq C_0^{1/2M}(M)\ve^{\del/2M}\\
 &\mbox{and}\,\,\,w_{2M}(\cL(\frac {H^\ve_x-\bar X_x}{\sqrt\ve}),\,\cL(G))\leq 2\hat C_0^{1/2M}(M)\ve^{\ve/2}.\nonumber
 \end{eqnarray}
 It is known (see, for instance, Theorem 2 in \cite{GS}) that $(\pi(\mu,\nu))^2\leq w_1(\mu,\nu)$
 but since we claim (\ref{2.13}) only for $M\geq 1$, we will provide below a slightly better estimate
 than what follows from this one.
  Let $\cX$ be the metric space of measurable paths $\gam:\,[0,T]\to\bbR^d$ with the uniform
metric $d(\gam,\tilde\gam)=\sup_{0\leq t\leq T}|\gam(t)-\tilde\gam(t)|$.
 \begin{corollary}\label{cor2.2} For any $\ve>0$,
\begin{equation}\label{2.14}
\pi(\cL(\frac {X^\ve_x-\bar X_x}{\sqrt\ve}),\,\cL(G))\leq C_0^{1/3}(M)\ve^{\del/3}\,\,\mbox{and}\,\,
\pi(\cL(\frac {H^\ve_x-\bar X_x}{\sqrt\ve}),\,\cL(G))\leq\hat C_0^{1/3}(M)\ve^{M/3}
\end{equation}
where $G$ is given by (\ref{1.4}) and, recall,
that (\ref{2.13}) depends only on distributions and not on specific choices of the stationary
process $\xi$ and of the Brownian motion $W$ as in Theorem \ref{thm2.1}.
\end{corollary}




We observe that the only place which forces us to have the estimate with only small fixed powers
of $\ve$ in (\ref{2.11}), (\ref{2.12}) and (\ref{2.13}), and not a growing with $M$ power of $\ve$,
is the strong approximation estimates of Theorem \ref{thm4.1} and Lemma \ref{lem4.4} below which cannot
be improved substantially within the current method.

Important classes of processes satisfying our conditions come from
dynamical systems. Let $F$ be a $C^2$ Axiom A diffeomorphism (in
particular, Anosov) in a neighborhood of an attractor or let $F$ be
an expanding $C^2$ endomorphism of a Riemannian manifold $\Om$ (see
\cite{Bow}), $f$ be either a H\" older continuous vector function or
a vector function which is constant on elements of a Markov partition and let $\xi(n)=
\xi(n,\om)=f(F^n\om)$. Here the probability
space is $(\Om,\cB,P)$ where $P$ is a Gibbs invariant measure corresponding
to some H\"older continuous function and $\cB$ is the Borel $\sig$-field. Let
$\zeta$ be a finite Markov partition for $F$, then we can take $\cF_{kl}$
 to be the finite $\sig$-algebra generated by the partition $\cap_{i=k}^lF^i\zeta$
 (or by $\cap_{i=k}^lF^{-i}\zeta$ in the non invertible case).
 In fact, we can take here not only H\" older continuous $f$'s but also indicators
of sets from $\cF_{kl}$. The conditions of Theorem \ref{thm2.1} allow all such functions
since the dependence of H\" older continuous functions on $m$-tails, i.e. on events measurable
with respect to $\cF_{-\infty,-m}$ or $\cF_{m,\infty}$, decays exponentially fast in $m$ and
the condition (\ref{2.6}) is even weaker than that.  A related class of dynamical systems
corresponds to $F$ being a topologically mixing subshift of finite type which means that $F$
is the left shift on a subspace $\Om$ of the space of one (or two) sided
sequences $\om=(\om_i,\, i\geq 0),\, \om_i=1,...,l_0$ such that $\om\in\Om$
if $\pi_{\om_i\om_{i+1}}=1$ for all $i\geq 0$ where $\Pi=(\pi_{ij})$
is an $l_0\times l_0$ matrix with $0$ and $1$ entries and such that $\Pi^n$
for some $n$ is a matrix with positive entries. Again, we have to take in this
case $f$ to be a H\" older continuous bounded function on the sequence space above.
 $P$ to be a Gibbs invariant measure corresponding to some H\" older continuous function
 and to define $\cF_{kl}$ as the finite $\sig$-algebra generated by cylinder sets
with fixed coordinates having numbers from $k$ to $l$. The
exponentially fast $\psi$-mixing, which is the strongest type of mixing among
mentioned above, is well known in these cases (see \cite{Bow}). Among other
dynamical systems with exponentially fast $\psi$-mixing we can mention also the Gauss map
$Fx=\{1/x\}$ (where $\{\cdot\}$ denotes the fractional part) of the
unit interval with respect to the Gauss measure and more general transformations generated
by $f$-expansions (see \cite{Hei}). Gibbs-Markov maps which are known to be exponentially fast
$\phi$-mixing (see, for instance, \cite{MN}) can be taken as $F$ with $\xi(n)=f\circ F^n$ as above.
Moreover, (\ref{2.2}) enables us to apply the results to $\alpha$ (and so also to $\beta$) mixing 
dynamical systems with a sufficiently fast decay of the $\alpha$ (or $\beta$) dependence coefficient,
among them some systems which can be represented via the Young tower construction.
Observe that in the above symbolic setups the assumption that $\xi(n,\om)=f(F^n\om)$ is $\cF_{nn}$-measurable
would mean that $f$ depends only on one coordinate or is constant on the elements of the basic
partition which is, of course, a very restrictive assumption.

\begin{remark}
We believe that a combination of methods from \cite{Bak04}, \cite{BK} and the present paper will yield a
version of Theorem \ref{thm2.1} for the specific fully coupled averaging setup of the form
\[
X^\ve((n+1)\ve)=X^\ve(n\ve)+\ve B(X^\ve(n\ve),\xi(n)),\,\,\,\xi(n+1)=T_{X^\ve(n\ve)}\xi(n),
\]
$X^\ve(0)=x,\,\xi(0)=y$
where $T_x,\, x\in\bbR^d$ is a $C^2$ depending on $x$ family of either $C^2$ expanding transformations
or $C^2$ Axiom A diffeomorphisms in a neighborhood of an attractor $\La_x,\, x\in \bbR^d$ so that
$\La_x$ corresponds to $T_x$. For more details about this setup we refer the reader to \cite{Bak04} and \cite{Ki09}.
Still, the treatment of the asymptotical behavior of the slow motion $X^\ve$ in this situation requires
substantial machinery from dynamical systems and the theory of perturbations which goes beyond the scope of
the present paper, and so this study will be left for another publication. We observe though that it is possible
to handle the fully coupled setup only when an appropriate model of slowly changing well mixing fast motions is
available which is not the case of a general stationary process $\xi$ considered in the present paper.
\end{remark}

\subsection{Continuous time case}\label{subsec2.2}

Here we start with a complete probability space $(\Om,\cF,P)$, a
$P$-preserving invertible transformation $\vt:\,\Om\to\Om$ and
a two parameter family of countably generated $\sig$-algebras
$\cF_{m,n}\subset\cF,\,-\infty\leq m\leq n\leq\infty$ such that
$\cF_{mn}\subset\cF_{m'n'}\subset\cF$ if $m'\leq m\leq n
\leq n'$ where $\cF_{m\infty}=\cup_{n:\, n\geq m}\cF_{mn}$ and
$\cF_{-\infty n}=\cup_{m:\, m\leq n}\cF_{mn}$. The setup includes
also a (roof or ceiling) function $\tau:\,\Om\to (0,\infty)$ such that
for some $\bar L>0$,
\begin{equation}\label{2.15}
\bar L^{-1}\leq\tau\leq\bar L.
\end{equation}
Next, we consider the probability space $(\hat\Om,\hat\cF,\hat P)$ such that $\hat\Om=\{\hat\om=
(\om,t):\,\om\in\Om,\, 0\leq t\leq\tau(\om),\, (\om,\tau(\om))=(\vt\om,0)\}$, $\hat\cF$ is the
restriction to $\hat\Om$ of $\cF\times\cB_{[0,\hat L]}$, where $\cB_{[0,\hat L]}$ is the Borel
$\sig$-algebra on $[0,\hat L]$ completed by the Lebesgue zero sets, and for any $\Gam\in\hat\cF$,
\[
\hat P(\Gam)=\bar\tau^{-1}\int_\Gam\bbI_\Gam(\om,t)dP(\om)dt\,\,\mbox{where}\,\,\bar\tau=\int\tau dP=E\tau,
\]
$E$ denotes the expectation on the space $(\Om,\cF,P)$ and $\hat E$ will denote the expectation on $(\hat\Om,\hat\cF,\hat P)$.
Finally, we introduce a vector valued stochastic process $\xi(t)=\xi(t,(\om,s))$, $-\infty<t<\infty,\, 0\leq s\leq\tau(\om)$ on $\hat\Om$ satisfying
\begin{eqnarray*}
&\xi(t,(\om,s))=\xi(t+s,(\om,0))=\xi(0,(\om,t+s))\,\,\mbox{if}\,\, 0\leq t+s<\tau(\om)\,\,\mbox{and}\\
&\xi(t,(\om,s))=\xi(0,(\vt^k\om,u))\,\,\mbox{if}\,\, t+s=u+\sum_{j=0}^k\tau(\vt^j\om)\,\,\mbox{and}\,\,
0\leq u<\tau(\vt^k\om).
\end{eqnarray*}
This construction is called in dynamical systems a suspension and it is a standard fact that $\xi$ is a stationary process on the probability space $(\hat\Om,\hat\cF,\hat P)$ and in what follows we will write also $\xi(t,\om)$ for $\xi(t,(\om,0))$.

We will assume that $X^\ve(t)=X^\ve(t,\om)$ considered as a process on $(\Om,\cF,P)$ solves the equation (\ref{1.1})
and the averaged motion $\bar X(t)$ is the solution of (\ref{1.3}) with $\bar B(x)=\hat EB(x,\xi(0,\om))$. Set
\begin{eqnarray*}
& b(x,\om)=\int_0^{\tau(\om)}B(x,\xi(s,\om))ds\,\,\mbox{and}\\
& \rho(a,n)=\sup_{x,m}\max_{i,j,k}\max\big(\|\tau\circ\vt^m-E(\tau\circ\vt^m|\cF_{m-n,m+n})\|_a,\,\|b_i(x,\xi(m))\\
&-E(b_i(x,\xi(m))|\cF_{m-n,m+n})\|_a,\,\|\frac {\partial b_i(x,\xi(m))}{\partial x_j}-E(\frac {\partial b_i(x,\xi(m))}{\partial x_j}|\cF_{m-n,m+n})\|_a,\nonumber\\
&\|\frac {\partial^2 b_i(x,\xi(m))}{\partial x_j\partial x_k}-E(\frac {\partial^2 b_i(x,\xi(m))}{\partial x_j\partial x_k}|\cF_{m-n,m+n})\|_a\big).\nonumber
\end{eqnarray*}
Since we assume that $B(x,\zeta)$ is twice differentiable in the first variable the last $\sup_x$ is still measurable. 
Observe also that $b(x,\cdot)\circ\vt^k$, $k\in\bbZ$ is a stationary sequences of random vectors.

Next, we consider the Gaussian process $G(t)$ and the diffusion $H^\ve(t)=H^\ve_x(t)$ given by
\begin{equation}\label{2.16}
dG(t)=\nabla\bar b(\bar X_x(\bar\tau t))G(t)dt+\sig(\bar X_x(\bar\tau t))dW(t),\, G(0)=0
\end{equation}
and
 \begin{equation}\label{2.17}
 dH^\ve(t)=\bar b(H^\ve(t))dt+\sqrt\ve\sig(H^\ve(t))dW,\, H^\ve_x(0)=H^\ve(0)=x,
 \end{equation}
 respectively, where $\bar b(x)=Eb(x,\cdot)$, $\sig^2(x)=A(x)=(a_{ij}(x))_{i,j=1,...,d}$ and
 \begin{eqnarray}\label{2.18}
 &a_{ij}(x)=\lim_{t\to\infty}\frac 1t\int_0^t\int_0^t\hat E(\hat B_i(x,\xi(s,\om))\hat B_j(x,\xi(u,\om)))dsdu\\
  &\lim_{n\to\infty}\frac 1n\sum_{k,l=0}^nE(b_i(x,\vt^k\om)-\tau(\vt^k\om)\bar B_i(x))(b_j(x,\vt^l\om)-\tau(\vt^l\om)\hat B_j(x)),
 \nonumber\end{eqnarray}
 $\hat B(x,\xi(s,\om))=B(x,\xi(s,\om))-EB(x,\xi(x,\om))$ and the limits in (\ref{2.18}) will be shown to exist in the same way as
 in (\ref{2.7}). As before we set also $\hat G^\ve=\bar X+\ve G$. Our results for this continuous time setup are the following.

 \begin{theorem}\label{thm2.4}
Suppose that the conditions (\ref{2.3}), with $b$ in place of $B$, and (\ref{2.6}) hold true with $\varpi$ defined with respect
to the $\sig$-algebras $\cF_{mn}$. Let a symmetric Lipschitz continuous
matrix $\sig(x)$ satisfying (\ref{2.8}) is fixed. Then for each $\ve>0$ the process $\xi(t),\, 0\leq t<\infty$ viewed on the basic 
probability space $(\Om,\cF,P)$ can be
redefined preserving its distributions on a richer probability space where there exists a standard $d$-dimensional Brownian motion
$W=W_\ve$ so that the slow motion $X_x^\ve$, the diffusion $H_x^\ve$ from (\ref{2.17}) and the Gaussian process $\hat G_x^\ve$
 constructed with these newly defined processes and having the same initial condition $X^\ve(0)=H^\ve(0)=\hat G^\ve(0)=x$, satisfy
\begin{equation}\label{2.19}
E\sup_{0\leq t\leq T}|X^\ve_x(t)-H_x^\ve(t/\bar\tau)|^{2M}\leq C_0(M)(\ve^{M+\del}+\ve^{(3M-4)/2})\quad\mbox{and}
\end{equation}
\begin{equation}\label{2.20}
E\sup_{0\leq t\leq T}|X^\ve_x(t)-\hat G_x^\ve(t/\bar\tau)|^{2M}\leq C_0(M)(\ve^{M+\del}+\ve^{(3M-4)/2}),
\end{equation}
for any integer $M\geq 1$ and some $C_0(M)>0$ which do not depend on $\ve$ while, again, we can take $\del=\frac 1{270d}$.
In particular,
\begin{equation}\label{2.21}
E\sup_{0\leq t\leq T}|\ve^{-1/2}(X^\ve(t)-\bar X_x(t))-G(t/\bar\tau)|^{2M}\leq C_0(M)(\ve^{\del}+\ve^{(M-4)/2}).
\end{equation}
Furthermore,
\begin{equation}\label{2.22}
E\sup_{0\leq t\leq T}|\ve^{-1/2}(H_x^\ve(t)-\bar X_x(t))-G(t/\bar\tau)|^{2M}\leq\hat C_0(M)(\ve^{M}+\ve^{(M-4)/2}),
\end{equation}
where $\hat C_0(M)$ does not depend on $\ve$, provided $G$ and $H_x$ are given by (\ref{1.4}) and (\ref{1.5}), respectively, with
the same Brownian motion $W$.
Corollary \ref{cor2.2} remains true with the same constants in this continuous setup, as well.
\end{theorem}

The proof of Theorem \ref{thm2.4} proceeds by reducing the problem to the corresponding limit theorems for certain discrete time
processes on the probability space $(\Om,\cF,P)$ given by the recurrence relations similar to (\ref{1.6}) but with certain (random)
vector field $b(x,\vt^n\om)$ in place of $B(x,\xi(n,\om))$ in (\ref{1.6}). We observe that though the former look
slightly more general than the latter, we still will be able to rely on results of Section \ref{sec3} and \ref{sec4} since we will use
there only the appropriate decay of the approximation coefficient $\rho$ and not any specific properties of the process $\xi$ except for
its stationarity which is replaced by the assumption that $\vt$ preserves the probability $P$ on $\Om$. Moreover, if $(\Om,\cF,P)$ is a
Lebesgue (standard probability) space then we can always represent $b(y,\vt^n\om)$ in the form $b(y,\zeta(n,\om))$
 where $\zeta(n,\om)=\zeta(\vt^n\om)$ is a real valued stationary process which fits into the discrete time setup
of Sections \ref{sec3} and \ref{sec4}.

The main application to dynamical systems we have here in mind is a $C^2$ Axiom A flow $F^t$ near an attractor which
using Markov partitions can be represented as a suspension over an exponentially fast $\psi$-mixing transformation so that
we can take $\xi(t)=f\circ F^t$ for a H\" older continuous function $f$ and the probability $P$ to be a Gibbs invariant
measure constructed by a H\" older continuous potential on the base of the Markov partition (see, for instance, \cite{BR}).
The space $\Om$ above is the union of bases of a Markov partition and it can be considered as a symbolic shift space with
a finite alphabet where $\sig$-algebras $\cF_{mn}$ are generated by cylinder sets.

Theorems \ref{thm2.1} and \ref{thm2.4} show that both $\hat G^\ve_x$ and $H^\ve_x$ provide the same
order of approximation of the slow motion $X_x^\ve$ but depending on circumstances it may be preferable to deal
with one or with the other. It is usually easier to study a specific Gaussian process than a general diffusion but,
on the other hand, the former can be considered on a differential manifold only in local coordinates while the later can be written
 there globally in an invariant form. Since the differential equation (\ref{1.1}) also can be considered on a maniford, we can obtain
 (\ref{2.19}) in the form $E\sup_{0\leq t\leq T}d^{2M}(X^\ve_x(t),\, H_x^\ve(t/\bar\tau))\leq C_0(M)(\ve^{M+\del}+\ve^{(3M-4)/2})$ where
 $d(\cdot,\cdot)$ is the distance on the manifold.
\begin{remark}\label{rem2.5}
In order to know whether the Gaussian process $G$ and the diffusion $H^\ve$ are non degenerate,
we have to know whether the matrix function $A(x)$ is positively definite for each $x$ which
means that the inner product $\langle A(x)\la,\la\rangle$ is positive for any $x,\la\in\bbR^d,\,\la\ne 0$.
Set $\gamma_{\la,x}(\xi)=\langle\la,\hat B(x,\xi)\rangle$. Then $\langle A(x)\la,\la\rangle$ is the limiting variance of the
normalized sum $n^{-1/2}\sum_{k=0}^n\gamma_{\la,x}(\xi(k))$ of one dimensional random variables. It is known (see, for instance,
Chapter 18 in \cite{IL}) that this variance is positive if and only if there exist no co-boundary
representation $\gamma_{\la,x}(\xi(0,\om))=f(\vt\om)-f(\om)$ for some $L^2$ function $f$.
\end{remark}

\subsection{Almost sure approximation}\label{subsec2.3}

Here we will restrict ourselves to the product case where $B(x,\xi)=\Sig(x)\xi$ where $\Sig(x)$ is a smooth $d\times d$
 matrix function and $\xi\in\bbR^d$. We assume now (\ref{2.5}), (\ref{2.6}),
\begin{equation}\label{2.23}
\sup_x\|\Sig(x)\|_{C^2}\leq L\,\,\,\mbox{and}\,\,\, \sup_n|\xi(n)|\leq L\,\,\mbox{a.s.}
\end{equation}
Applying estimates of Lemma \ref{lem3.5} to our situation we see that the limit
\[
\vs_{ij}=\lim_{n\to\infty}\frac 1n\sum_{0\leq k,l\leq n}E(\xi_i(k)\xi_j(l))
\]
exists and
\[
a_{ij}(x)=\sum_{k,l=1}^d\Sig_{ik}(x)\vs_{kl}\Sig_{lj}(x),\,\,\mbox{i.e.}\,\, A(x)=\Sig(x)\vs\Sig^*(x),
\]
and so $\sig(x)=\Sig(x)\vs^{1/2}$.
\begin{theorem}\label{thm2.6}
The stationary process $\xi(n),\,-\infty <n<\infty$ can be redefined preserving its distributions on a richer
probability space where there exists a Brownian motion $\cW$ with the covariance matrix $\vs$ (at the time 1) such
that the slow motion $X^\ve$ solving (\ref{1.1}) with $B(x,\xi)=\Sig(x)\xi$ and the redefined process $\xi(n)$ together
 with the Gaussian process $G(t)=G_\ve(t)$ determined by (\ref{1.4}) with $W(t)=W_\ve(t)$ such that $\vs^{1/2}W(t)=
 \sqrt\ve\cW(t/\ve)$ satisfy
\begin{equation}\label{2.24}
\sup_{0\leq t\leq T}|\ve^{-1/2}(X^\ve_x(t)-\bar X_x(t))-G_\ve(t)|=O(\ve^\del)\quad\mbox{(a.s.)}
\end{equation}
for some $\del>0$ which can be estimated from the proof.
\end{theorem}

\subsection{Law of iterated logarithm}\label{subsec2.4}

Here we continue working with the product case $B(x,\xi)=\Sig(x)\xi$. Let $\cC_d[0,T]$ be the Banach space of continuous
vector functions $\vf=(\vf_1,...,\vf_d)$ on the interval $[0,T]$ with the supremum norm $\|\vf\|_{[0,T]}=\max_{1\leq k\leq d}
\sup_{0\leq t\leq T}|\vf_k(t)|$.
It is easy to see that for each $\vf\in\cC_d[0,T]$ there exists the unique $\Phi(\vf)\in\cC_d[0,T]$ such that
\begin{equation}\label{2.25}
\Phi(\vf)(t)=\int_0^t\nabla\bar B(\bar X_x(s))\Phi(\vf)(s)ds+\vf(t)
\end{equation}
and the map $\Phi:\,\cC_d[0,T]\to\cC_d[0,T]$ is continuous. Next, introduce another continuous map $\Psi:\,\cC_d[0,T]\to\cC_d[0,T]$
defined by
\begin{equation}\label{2.26}
\Psi(\vf)(t)=\sig(\bar X_x(t))\vf(t)-\int_0^t\nabla\sig(\bar X_x(u))\bar B(\bar X_x(u))\vf(u)du
\end{equation}
where for any smooth $d\times d$ matrix function $\sig$ and a vector $\eta=(\eta_1,...,\eta_d)$ we denote by $\nabla\sig(y)\eta$
 the $d\times d$ matrix function with
$(\nabla\sig(y)\eta)_{ij}=\sum_{1\leq k\leq d}\frac {\partial\sig_{ij}(y)}{\partial y_k}\eta_k$. Let $\cK$ be the compact set
of absolutely continuous vector functions $\vf\in\cC_d[0,T]$ such that $\int_0^T(\frac {d\vf(s)}{ds})^2ds\leq 1$. We will derive the
following Strassen's type law of iterated logarithm.

\begin{theorem}\label{thm2.7} Assume that (\ref{2.23}) holds true and that $B(x,\xi)=\Sig(x)\xi$. Then with probability one
the set of limit points in the supremum norm as $\ve\to 0$ of random functions
\[
\frac {X^\ve_x(t)-\bar X_x(t)}{\sqrt {2\ve\log\log\frac 1\ve}},\quad t\in[0,T]
\]
coincides with the compact set $\Phi\Psi(\cK)$.
\end{theorem}

 \section{Auxiliary estimates}\label{sec3}\setcounter{equation}{0}
  \subsection{General lemmas}\label{subsec3.1}
 First, we will formulate three general results which will be used throughout this
 paper. The following lemma is well known (see, for instance, Lemma 1.3.10 in \cite{HK}).
 \begin{lemma}\label{lem3.1}
  Let $G(x,\om)$ and $H(x,\om)$ be a measurable function on the space $(\bbR^d\times\Om,\,\cB\times\cF)$,
 where $\cB$ is the Borel $\sig$-algebra, such that for each $x\in\bbR^d$ the function $G(x,\cdot)$
 is measurable with respect to a $\sig$-algebra $\cG\subset\cF$ and $H(x,\cdot)$ is measurable with respect
 to another $\sig$-algebra $\cH\subset\cF$. 
 Then for any $x,y\in\bbR^d$,
 \begin{eqnarray}\label{3.1}
  &|E(G(x,\cdot)H(y,\cdot))-EG(x,\cdot)EH(y,\cdot)|\leq \| G(x,\cdot)\|_r\| H(y,\cdot)\|_q\varpi_{r,p}(\cH,\cG)\\
  &\leq\| G(x,\cdot)\|_\infty\| H(y,\cdot)\|_\infty\varpi_{r,p}(\cH,\cG)\nonumber
  \end{eqnarray}
where $1\leq p,q,r\leq\infty,\,\frac 1p+\frac 1q\leq 1$ and, as before, $\|\cdot\|_s$ is the $L^s$-norm of random variables.
 \end{lemma}
 \begin{proof} By the H\" older inequality,
 \begin{eqnarray*}
 &|E(G(x,\cdot)H(y,\cdot))-EG(x,\cdot)EH(y,\cdot)|=|E\big((E(G(x,\cdot)|\cH)-EG(x,\cdot))H(y,\cdot)\big)|\\
 &\leq\|H(y,\cdot)\|_q\| E(G(x,\cdot)|\cH)-EG(x,\cdot)\|_p\leq\| G(x,\cdot)\|_r\| H(y,\cdot)\|_q\varpi_{r,p}(\cH,\cG),
 \end{eqnarray*}
 as required.
 \end{proof}
 
  We will employ several times the following general moment estimate which appeared as Lemma 3.2.5 in \cite{HK} for random variables and was extended to random vectors in
  Lemma 3.4 from \cite{Ki24}. We observe that neither proof of Lemma 3.4 from \cite{Ki24} nor the the proof of Lemma 3.2.5 in \cite{HK} rely on any weak dependence assumptions,
  and so the arguments there are still valid no matter whether the dependence conditions are expressed by the coefficient $\vt$ as in \cite{HK} and \cite{Ki24} or by the coefficient
  $\varpi$ as here.
   \begin{lemma}\label{lem3.2}
   Let $(\Om,\cF,P)$ be a probability space,  $\cG_j,\, j\geq 1$ be a filtration of $\sig$-algebras and $\eta_j,\, j\geq 1$ be
   a sequence of random $d$-dimensional vectors such that $\eta_j$
   is $\cG_j$-measurable, $j=1,2,...$. Suppose that for some integer $M\geq 1$,
   \[
   A_{2M}=\sup_{i\geq 1}\sum_{j\geq i}\| E(\eta_j|\cG_i)\|_{2M}<\infty
   \]
   where $\|\eta\|_p=(E|\eta|^p)^{1/p}$ and $|\eta|$ is the Euclidean norm of a (random) vector $\eta$.
   Then for any integer $n\geq 1$,
   \begin{equation}\label{3.2}
   E|\sum_{j=1}^n\eta_j|^{2M}\leq 3(2M)!d^MA_{2M}^{2M}n^M.
   \end{equation}
   \end{lemma}

   In order to obtain uniform moment estimates required by Theorem \ref{thm2.1} we will need the   
    following general estimate which appeared as Lemma 3.7 in \cite{Ki24} where in the last
   inequality below we use also Lemma \ref{lem3.2} above. We observe again that the latter
   lemma was not based on any dependence assumptions, and so its proof is valid in our circumstances,
   as well.
 \begin{lemma}\label{lem3.3} Let $\eta_1,\eta_2,...,\eta_N$ be random $d$-dimensional vectors and
 $\cH_1\subset\cH_2\subset...\subset\cH_N$ be a filtration of $\sig$-algebras such that $\eta_m$ is
 $\cH_m$-measurable for each $m=1,2,...,N$. Assume also that $E|\eta_m|^{2M}<\infty$ for some $M\geq 1$
 and each $m=1,...,N$. Set $S_m=\sum_{j=1}^m\eta_j$. Then
 \begin{eqnarray}\label{3.3}
 &E\max_{1\leq m\leq N}|S_m|^{2M}\leq 2^{2M-1}\big((\frac {2M}{2M-1})^{2M}E|S_N|^{2M}\\
 &+E\max_{1\leq m\leq N-1}|\sum^N_{j=m+1}E(\eta_j|\cH_m)|^{2M}\big)\leq 2^{2M-1}A^{2M}_{2M}
 (3(2M)!d^MN^M+N).\nonumber
 \end{eqnarray}
 \end{lemma}

 We will need also the following moment estimates for sums and iterated sums which appeared as Lemma 3.4 in \cite{FK}
 under the $\vt$-mixing condition and as Lemma 3.2 in \cite{Ki24+} were it was proved under the same general dependence 
 conditions as here.
  \begin{lemma}\label{lem3.4}
  Let $\eta(k),\,\zeta_k(l),\, k=0,1,2,...,\, l=0,...,k$ be two sequences of random variables on the
  probability space $(\Om,\cF,P)$ such that for all $k,l,n\geq 0$,
 \begin{eqnarray*}
 & \|\eta(k)-E(\eta(k)|\cF_{k-n,k+n})\|_K,\,  \|\zeta_k(l)-E(\zeta_k(l)|\cF_{l-n,l+n})\|_K\leq\rho(K,n),\\
 &\|\eta(k)\|_K,\,\|\zeta_k(l)\|_K\leq\tilde\gam_K<\infty\quad\mbox{and}\quad E\eta(k)=E\zeta(k)=0
  \end{eqnarray*}
  where the $\sig$-algebras $\cF_{kl}$ are the same as in Section \ref{sec2}. Then for any $N,M\geq 1$,
  \begin{equation}\label{3.4}
  E\max_{1\leq n\leq N}(\sum_{k=0}^{n}\eta(k))^{2M}\leq C_1^\eta(M)N^{M}
  \end{equation}
  and
  \begin{equation}\label{3.5}
  E\max_{1\leq n\leq N}(\sum_{k=1}^{n}\eta(k)\sum_{l=0}^{k-1}\zeta_k(l))^{2M}\leq C_1^{\eta,\zeta}(M)N^{2M}
  \end{equation}
  where $C_1(M)>0$ depends only on $\rho,\varpi, M, K$ and $\tilde\gam_K$ but it does not depend on $N$ and on the 
  sequences $\eta(n),\zeta_k(n),\, n\geq 1$ themselves.
  \end{lemma}

 \subsection{Approximations}\label{susec3.2}
 The following result shows that the definition (\ref{2.7}) is legitimate and it estimates
 also the speed of convergence in this limit which will be needed for comparison of characteristic
 functions later on.

 \begin{lemma}\label{lem3.5}
 For each $x\in\bbR^d$ the limit (\ref{2.7}) exists and for all $m,n\geq 0$ and $i,j=1,...,d$,
 \begin{equation}\label{3.7}
 |na_{ij}(x)-\sum_{k=0}^{n-1}\sum_{l=0}^{n-1}a_{ij}(x,k,l)|\leq\hat L=  4L\sum_{k=0}^\infty\sum_{l=k}^\infty(\rho(K,l/3)+L\varpi_{K,2M}(l/3)).
 \end{equation}
 Moreover, $a_{ij}(x)$ is twice differentiable for $i,j=1,...,d$ and all $x\in\bbR^d$,
 \begin{equation}\label{3.8}
 \|( a_{ij})\|_{C^2}\leq\tilde L=4Ld^2\sum_{l=0}^\infty(\rho(K,l/3)+L\varpi_{K,2M}(l/3)\,\,\mbox{and}
 \,\,\sup_x|\sig(x)|\leq\sqrt{\tilde Ld}
 \end{equation}
 where $\sig(x)$ is the Lipschitz continuous square root of the matrix $A(x)=(a_{ij}(x))$ which
 exists in view of \cite{Fre} and \cite{SV}.
 Furthermore, for all $0\leq s<t\leq T$ and $\ve>0$,
 \begin{equation}\label{3.9}
 |\ve^{-1}\int_s^ta_{ij}(\bar X_x(u))du-\sum_{s/\ve\leq k,l\leq t/\ve}a_{ij}(\bar X_x(k\ve),
 \bar X_x(l\ve),k,l)|\leq C_2(T)\ve^{-2/3}
 \end{equation}
 where $C_2(T)>0$ does not depend on $\ve$.
 \end{lemma}
 \begin{proof}
 By Lemma \ref{lem3.1} for any $m,n$, $x,y\in\bbR^d$ and $i,j=1,...,d$,
 \begin{eqnarray}\label{3.10}
 &|E(\hat B_i(x,\xi(m)),\hat B_j(y,\xi(n)))|\leq 4L\rho(K,\frac 13|m-n|)\\
 &+|E\big(E(\hat B_i(x,\xi(m))|\cF_{m-\frac 13|m-n|,m+\frac 13|m-n|})\nonumber\\
 &\times E(\hat B_j(y,\xi(n))|\cF_{n-\frac 13|m-n|,n+\frac 13|m-n|})\big)|\nonumber\\
 &\leq 4L(\rho(K,\frac 13|m-n|)+L\varpi_{K,2M}(\frac 13|m-n|)).\nonumber
 \end{eqnarray}
 Hence, by the stationarity of the process $\xi$,
 \begin{eqnarray*}
 &\lim_{n\to\infty}\sum_{l=0}^{n-1}a_{ij}(x,k,l)=\sum_{l=0}^{k-1}a_{ij}(x,k,l)+\sum_{l=k}^\infty  a_{ij}(x,k,l)\\
 &=\sum_{m=1}^ka_{ij}(x,m,0)+\sum_{l=0}^\infty a_{ij}(x,0,l),
 \end{eqnarray*}
 and so
 \[
 a_{ij}(x)=\lim_{n\to\infty}\frac 1n\sum_{k=0}^{n-1}\sum_{l=0}^{n-1}a_{ij}(x,k,l)
 =\sum_{k=1}^\infty a_{ij}(x,k,0)+\sum_{l=0}^\infty a_{ij}(x,0,l)
 \]
 where all limits exist and infinite sums converge in view of (\ref{2.6}) and (\ref{3.10}).
 Thus, by (\ref{3.10}),
 \begin{eqnarray*}
 &|a_{ij}(x)-\frac 1n\sum_{k=0}^{n-1}\sum_{l=0}^{n-1}a_{ij}(x,k,l)|=|\frac  1n\sum_{k=0}^{n-1}\sum_{m=k+1}^\infty a_{ij}(x,m,0)\\
 &+\frac 1n\sum_{k=0}^{n-1}\sum_{l=n-k}^\infty a_{ij}(x,0,l)|\leq\hat Ln^{-1}
 \end{eqnarray*}
 which gives (\ref{3.7}).

 Next, $|\Del_k|^{-1}|\hat B(x+\Del_k,\xi)-\hat B(x,\xi)|\leq 2L$ which together with the dominated
 convergence theorem gives that
 \[
 0=\frac \partial{\partial x_k}E\hat B(x,\xi)=E\frac \partial{\partial x_k}\hat B(x,\xi).
 \]
 Similarly,
 \[
 E\frac {\partial^2}{\partial x_j\partial x_k}\hat B(x,\xi)=0.
 \]
 This together with (\ref{2.5}) enables us to argue in the same way as above to conclude that
 \[
 \frac {\partial a_{ij}(x)}{\partial x_k}=\sum_{m=1}^\infty\frac {\partial a_{ij}(x,m,0)}{\partial x_k}
 +\sum_{m=0}^\infty\frac {\partial a_{ij}(x,0,m)}{\partial x_k}
 \]
 and
 \[
 \frac {\partial^2 a_{ij}(x)}{\partial x_k\partial x_l}=\sum_{m=1}^\infty\frac {\partial^2
  a_{ij}(x,m,0)}{\partial x_k\partial x_l}
 +\sum_{m=0}^\infty\frac {\partial^2 a_{ij}(x,0,m)}{\partial x_k\partial x_l}
 \]
 and these series converge absolutely since similarly to (\ref{3.7}) we see that
 \[
 \max(|\frac {\partial a_{ij}(x,m,0)}{\partial x_k}|,\, |\frac {\partial^2
  a_{ij}(x,m,0)}{\partial x_k\partial x_l}|)\leq 4L(\rho(K,m/3)+L\varpi_{K,2M}(m/3)).
  \]
  Observe that for any vector $y\in\bbR^d$,
  \[
  \tilde Ld|u|^2\geq\langle A(x)y,y\rangle=\langle\sig(x)y,\sig(x)y\rangle=|\sig(x)y|^2,
  \]
  and so $|\sig(x)|\leq\sqrt {\tilde Ld}$ completing the proof of (\ref{3.8}).

  Next, by (\ref{3.10}) and the stationarity of the process $\xi$, for any $0\leq s<t\leq T$ and
   an integer $n\geq 1$,
   \begin{eqnarray*}
   &|\sum_{s/\ve\leq k,l\leq t/\ve}a_{ij}(\bar X_x(k\ve),\bar X_x(l\ve),k,l)\\
   &-\sum_{m=0}^{n-1}\sum_{\ve^{-1}(s+m(t-s)/n)\leq k,l\leq\ve^{-1}(s+(m+1)(t-s)/n)}a_{ij}
   (\bar X_x(k\ve),\bar X_x(l\ve),k,l)|\\
   &\leq 8Ln\sum_{m=0}^\infty\sum_{k=m}^\infty(\rho(K,k/3)+L\varpi_{K,2M}(k/3))=2\hat Ln.
   \end{eqnarray*}
   By (\ref{2.3}) we have also
   \begin{eqnarray*}
   &|\sum_{\ve^{-1}(s+m(t-s)/n)\leq k,l\leq\ve^{-1}(s+(m+1)(t-s)/n)}a_{ij}
   (\bar X_x(k\ve),\bar X_x(l\ve),k,l)\\
   &-\sum_{\ve^{-1}(s+m(t-s)/n)\leq k,l\leq\ve^{-1}(s+(m+1)(t-s)/n)}a_{ij}
   (\bar X_x(s+m(t-s)/n),k,l)|\\
   &\leq 2L^3(\frac {t-s}n)^3\ve^{-2}.
   \end{eqnarray*}
   Now, by (\ref{3.8}),
   \begin{eqnarray*}
   &|\int_s^ta_{ij}(\bar X_x(u))du-\frac {t-s}n\sum_{m=0}^{n-1}a_{ij}(\bar X_x(s+m\frac {t-s}n))|\\
   &\leq    8Ld^2n(\frac {t-s}n)^2\sum_{l=0}^\infty(\rho(K,l/3)+L\varpi_{K,2M}(l/3)).
   \end{eqnarray*}
   Finally, by (\ref{3.7}) and the stationarity of the process $\xi$,
   \begin{eqnarray*}
   &|\ve^{-1}\frac {t-s}na_{ij}(\bar X_x(s+m\frac {t-s}n))\\
   &-\sum_{\ve^{-1}(s+m(t-s)/n)\leq k,l\leq\ve^{-1}(s+(m+1)(t-s)/n)}a_{ij}(\bar    X_x(s+m(t-s)/n),k,l)|\leq\hat L.
   \end{eqnarray*}
   Combining four last inequalities we derive (\ref{3.9}) taking $n=[\ve^{-2/3}]$.
 \end{proof}

  The first step in the proof of estimates of Theorems \ref{thm2.1} and \ref{thm2.4}
  is the following result which is similar to Lemma 3.1 in \cite{Ki03}.
  \begin{lemma}\label{lem3.6} For all $\ve,T>0$, $M\geq 1$ and $x\in\bbR^d$,
  \begin{eqnarray}\label{3.11}
  &\sup_{0\leq t\leq T}|X^\ve_x(t)-H^\ve_x(t)|\leq e^{LT}\big(\sqrt\ve\sup_{0\leq t
  \leq T}|\sqrt\ve S^\ve(t)\\
  &-\int_0^t\sig(\bar X_x(s)dW(s)|+L\ve+\sum^4_{i=1}\sup_{0\leq t\leq T}|R^\ve_i(t)|\big)
  \nonumber\end{eqnarray}
  and
  \begin{equation}\label{3.12}
  \sup_{0\leq t\leq T}|H^\ve_x(t)-\hat G^\ve_x(t)|\leq e^{LTd}
  \sum^5_{i=4}\sup_{0\leq t\leq T}|R^\ve_i(t)|,
  \end{equation}
  where $\sig^2(x)=A(x)$, the stochastic integral is meant in the It\^ o sense and
  \[
  S^\ve(t)=\sum_{0\leq k< [t/\ve]}\big (B(\bar X_x(k\ve),\xi(k))-\bar B(\bar X_x(k\ve))\big),
  \]
  \[
  R^\ve_1(t)=\int_0^t(\nabla B(\bar X_x([s/\ve]\ve),\xi([s/\ve]))-\nabla\bar B(\bar   X_x([s/\ve]\ve)))(X^\ve_x(s)-\bar X_x([s/\ve]\ve))ds,
  \]
  \begin{eqnarray*}
  &R_2^\ve(t)=\int_0^t\langle C(\bar X_x([s/\ve]\ve)+\te(X^\ve_x(s)-\bar   X_x([s/\ve]\ve)),\xi([s/\ve]))(X^\ve_x(s)-\bar X_x([s/\ve]\ve)),\\
  &(X^\ve_x(s)-\bar X_x([s/\ve]\ve))\rangle ds,
  \end{eqnarray*}
  \begin{eqnarray*}
  &R_3(t)=\int_0^t\langle D(\bar X_x([s/\ve]\ve)+\tilde\te(X^\ve_x(s)-\bar   X_x([s/\ve]\ve)))(X^\ve_x(s)-\bar X_x([s/\ve]\ve)),\\
  &(X^\ve_x(s)-\bar X_x([s/\ve]\ve))\rangle ds,
  \end{eqnarray*}
  \[
  R_4(t)=\sqrt\ve\int_0^t(\sig(H^\ve_x(s))-\sig(\bar X_x(s)))dW(s),
  \]
  \[
  R^\ve_5(t)=\int_0^t\langle D(\bar X_x(s)+\tilde\te(Y^\ve_x(s)-\bar X_x(s)))(H^\ve_x(s)-\bar X_x(s)),
  (H^\ve_x(s)-\bar X_x(s))\rangle ds
  \]
 with $\nabla B(z,\eta)=(\frac {\partial B_i(z,\eta)}{\partial z_j}),$ $\nabla\bar B(z)=(\frac  {\partial B_i(z)}{\partial z_j}),$
 $\langle\cdot,\cdot\rangle$ denoting the inner product in $\bbR^d$,  $C(z,\eta)=\nabla^2_zB(z,\eta)=(\nabla^2_zB_i(z,\eta))$, and
 $D(z)=\nabla^2_z\bar B(z)).$
   \end{lemma}
   \begin{proof}
   By the Taylor formula,
   \begin{eqnarray}\label{3.13}
   &X_x^\ve(t)=x+\int_{0}^{[t/\ve]\ve}B(X^\ve_x(s),\xi([s/\ve]\ve))ds\\
   &=x+\int_0^{[t/\ve]\ve}B(\bar X_x([s/\ve]\ve),\xi([s/\ve]))ds\nonumber\\
   &+\int_0^{[t/\ve]\ve}\nabla B(\bar X_x([s/\ve]\ve),\xi([s/\ve]))(X_x^\ve(s)-\bar X_x([s/\ve]\ve))ds +R^\ve_2([t/\ve]\ve)\nonumber\\
   &=x+\int_0^{[t/\ve]\ve}\bar B(\bar X_x([s/\ve]\ve))ds+\int_0^{[t/\ve]\ve}\nabla\bar B(\bar X_x([s/\ve]\ve))(X_x^\ve(s)\nonumber\\
   &-\bar X_x([s/\ve]\ve))ds+\ve S^\ve(t)+R^\ve_1([t/\ve]\ve)+R^\ve_2([t/\ve]\ve)=
   x+\int_0^t\bar B(X^\ve_x(s)ds\nonumber\\
   &-\int_{[t/\ve]\ve}^t\bar B(X^\ve_x(s))ds
   +\ve S^\ve(t)+R^\ve_1([t/\ve]\ve)+R^\ve_2([t/\ve]\ve)-R^\ve_3([t/\ve]\ve)\nonumber
   \end{eqnarray}
   and
   \begin{eqnarray*}
   &H_x^\ve(t)=x+\int_0^t\bar B(H_x^\ve(s))ds+\sqrt\ve\int_0^t\sig(\bar X_x(s))dW(s)+R^\ve_4(t)=x+
   \int_0^t\bar B(\bar X_x(s))ds\\
   &+\int_0^t\nabla\bar B(\bar X_x(s))(H_x^\ve(s)-\bar X_x(s))ds+\sqrt\ve\int_0^t\sig(\bar X_x(s))dW(s)+R^\ve_4(t)+R^\ve_5(t).
   \end{eqnarray*}
   Hence, by (\ref{2.3}) and (\ref{3.13}),
   \begin{eqnarray*}
  & \sup_{0\leq s\leq t}|X^\ve_x(s)-H^\ve_x(s)|\leq L\int_0^t\sup_{0\leq u\leq   s}|X^\ve_x(u)-H^\ve_x(u)|ds\\
  &+\sqrt\ve\sup_{0\leq t\leq T}|\sqrt\ve S^\ve(t)-\int_0^t\sig(\bar   X_x(s))dW(s)|+L\ve+\sum_{i=0}^4\sup_{0\leq t\leq T}|R^\ve_i(t)|.
  \end{eqnarray*}
  Employing here the Gronwall inequality, we derive (\ref{3.11}). Similarly, by (\ref{1.4}) and the above Taylor representation of $H^\ve_x(t)$,
  \begin{eqnarray*}
  &\sup_{0\leq s\leq t}|H^\ve_x(s)-\bar X_x(s)-\sqrt\ve G(s)|\\
  &\leq Ld\int_0^t\sup_{0\leq u\leq s}|H^\ve_x(u)-\bar X_x(u)
  -\sqrt\ve G(u)|ds+\sup_{0\leq t\leq T}(|R^\ve_4(t)|+|R^\ve_5(t)|).
  \end{eqnarray*}
  Again, employing the Gronwall inequality we obtain (\ref{3.12}).
    \end{proof}

    Next, we estimate directly $R^\ve_i,\, i=1,2,...,5$ while the more significant
    first summand in the right hand side of (\ref{3.11}) will be treated in the next section
    by the strong approximations machinery.
    \begin{lemma}\label{lem3.7}
     (i) For any $T>0$ and an integer $M>0$, there exists a constant $C_{T}(M)>0$ such that
      for all $\ve>0$,
    \begin{equation}\label{3.14}
    E\sup_{0\leq t\leq T}|R^\ve_1(t)|^{2M}\leq C_{T}(M)\ve^{2M}
    \end{equation}
    and
    \begin{equation}\label{3.15}
     E\sup_{0\leq t\leq T}|R^\ve_i(t)|^{2M}\leq\check C_T(M)\ve^{2M}\quad\mbox{for} \,\, i=2,3.
     \end{equation}
     where $\check C_T(M)=C_T(M)L^{2M}d^{3M}T^{3M}e^{2MdLT}$.

     (ii) Suppose that $\sig$ in (\ref{1.4}) and (\ref{1.5}) satisfies for all $x,y\in\bbR^d$,
     \begin{equation}\label{3.16}
     |\sig(x)|\leq C_\sig\quad\mbox{and}\quad |\sig(x)-\sig(y)|\leq C_\sig|x-y|
     \end{equation}
     where $C_\sig>0$ is a constant and $|\cdot|$ is the Euclidean norm for vectors or matrices. Then
     \begin{equation}\label{3.17}
     E\sup_{0\leq t\leq T}|R_4^\ve(t)|^{2M}\leq 2^{4M}e^{2MLT}C_\sig^{4M}T^{2M}M^{6M}(2M-1)^{-2}\ve^{2M}
     \end{equation}
     and
     \begin{equation}\label{3.18}
     E\sup_{0\leq t\leq T}|R_5^\ve(t)|^{2M}\leq 4^{5M}L^{2M}C_\sig^{4M}T^{4M-1}d^{6M}e^{4MLT}
     M^{6M}(4M-1)^{-2M}\ve^{2M}.
     \end{equation}
     \end{lemma}
     \begin{proof} (i) Set
     \[
     \hat X_x^\ve(n\ve)=x+\ve\sum_{k=0}^{n-1}\bar B(\bar X_x(k\ve))=x+\int_9^{n\ve}
     \bar B(\bar X_x([s/\ve]\ve)ds.
     \]
     Then by (\ref{1.2}), (\ref{1.3}) and (\ref{2.3}) for each $n\leq T/\ve$,
     \begin{eqnarray*}
     &|\bar X_x(n\ve)-\hat X_x^\ve(n\ve)|\leq\int_0^{n\ve}|\bar B(\bar X_x(s)-\bar B(\bar X_x
     ([s/\ve]\ve))|ds\\
     &\leq L\int_0^{n\ve}|\int^s_{[s/\ve]\ve}\bar B(\bar X_x(u))du|ds\leq L^2T\ve.
     \end{eqnarray*}
     Again, by (\ref{2.3}),
     \begin{eqnarray*}
     &| X^\ve_x(n\ve)-\hat X_x(n\ve)|\leq\ve|\sum_{k=0}^{n-1}(B(X^\ve_x(k\ve),\xi(k))-
     B(\hat X^\ve_x(k\ve),\xi(k))|\\
     &+\ve|\sum_{k=0}^{n-1}(B(\hat X^\ve_x(k\ve),\xi(k))- B(\bar X_x(k\ve),\xi(k))|+
     \ve|\sum_{k=0}^{n-1}(B(\bar X_x(k\ve),\xi(k))\\
     &-\bar B(\bar X_x(k\ve))|\leq\ve L\sum_{k=0}^{n-1}|X^\ve_x(k\ve)-\hat X_x(k\ve)|+L^3T^2\ve\\
     &+ \ve|\sum_{k=0}^{n-1}(B(\bar X_x(k\ve),\xi(k))-E B(\bar X_x(k\ve),\xi(k))|
     \end{eqnarray*}
     since $\bar B(x)=EB(x,\xi(k))$.

     By the discrete time Gronwall inequality (see \cite{Cla}) it follows that for all $n\leq T/\ve$,
     \begin{eqnarray*}
    &| X^\ve_x(n\ve)-\hat X_x(n\ve)|\leq e^{LT}\big(L^3T^2\ve\\
    &+\ve\max_{1\leq n\leq T/\ve} |\sum_{k=0}^{n-1}( B(\bar X_x(k\ve),\xi(k))-
    E B(\bar X_x(k\ve),\xi(k)))|\big).
     \end{eqnarray*}
    Next, we apply Lemma \ref{lem3.4} with
    \[
    \eta_k=B_i(\bar X_x(k\ve),\xi(k))-EB_i(\bar X_x(k\ve),\xi(k)),\, i=1,2,...,d
    \]
    and $\tilde L=2L$ to obtain that
    \[
    E\max_{0\leq n\leq T/\ve}|X^\ve_x(n\ve)-\hat X^\ve_x(n\ve)|^{2M}\leq 2^{2M-1}e^{2MLT}d^{2M}
    (L^{6M}T^{4M}\ve^{2M}+C_1(M)T^M\ve^M).
    \]
    Combining this inequality with the above estimate for $|\bar X_x(n\ve)-\hat X_x^\ve(n\ve)|$ and
    taking into account that by (\ref{1.1})--(\ref{1.3}) and (\ref{2.3}),
    \[
    \sup_{0\leq s\leq T}\sup_{0\leq u\leq\ve}|X_x^\ve(s+u)-X_x^\ve(s)|\leq L\ve\,\,\mbox{and}\,\,
    \sup_{0\leq s\leq T}\sup_{0\leq u\leq\ve}|\bar X_x(s+u)-\bar X_x(s)|\leq L\ve
    \]
    we obtain that
    \begin{equation}\label{3.19}
    E \sup_{0\leq t\leq T}|X^\ve_x(t)-\bar X_x(t)|^{2M}\leq\hat C_T(M)\ve^M
    \end{equation}
    where $\hat C_T(M)>0$ does not depend on $\ve$. Applying (\ref{3.19}) with $2M$ in
    place of $M$ and taking into account (\ref{2.3}) we obtain that for $i=2,3$,
    \[
    E\sup_{0\leq t\leq T}|R^\ve_i(t)|^{2M}\leq L^{2M}T^{2M}E\max_{0\leq n\leq T/\ve}
    |X^\ve_x(n\ve)-\bar X_x(n\ve)|^{4M}\leq\hat C_T(2M)\ve^{2M}
    \]
    yielding (\ref{3.15}).

    In order to estimate $R_1^\ve$ we introduce the process $Z^\ve$ defined recursively for
    $n\geq 1$ by
    \begin{eqnarray}\label{3.20}
   & Z^\ve(n\ve)=\ve S^\ve(n\ve)+\ve\sum_{k=0}^{n-1}\nabla\bar B(\bar X_x(k\ve))Z^\ve(k\ve)\\
   &=\ve(S^\ve(n\ve)-S^\ve((n-1)\ve))+(I+\ve\nabla\bar B(\bar X_x((n-1)\ve)))Z^\ve((n-1)\ve)
   \nonumber\end{eqnarray}
   with $Z^\ve(0)=0$, where $I$ is the $d\times d$ identity matrix and $\nabla\bar B(x)$ is the
   matrix $(\frac {\partial\bar B_i(x)}{\partial x_j})$. It follows that
   \begin{equation}\label{3.21}
   Z^\ve(n\ve)=\ve\sum_{k=1}^{n}K^\ve(n,k)(S^\ve(k\ve)-S^\ve((k-1)\ve))
   \end{equation}
   where $K^\ve(n,n)=I$ and for $k<n$,
   \[
   K^\ve(n,k)=\prod_{l=1}^{n-k}(I+\ve\nabla\bar B(\bar X_x((n-l)\ve)))
   \]
   where we multiply the matrices from the right.

   Next, we consider the difference $U^\ve(n\ve)=X^\ve_x(n\ve)-\bar X_x(n\ve)-Z^\ve(n\ve)$
   which by (\ref{3.13}) and (\ref{3.20}) satisfies
   \[
   U^\ve(n\ve)-\ve\sum_{k=0}^{n-1}\nabla B(\bar X_x(k\ve),\xi(k))U^\ve(k\ve)=V^\ve(n\ve)+R^\ve_2(n\ve)
   \]
   where
   \[
   V^\ve(n\ve)=\ve\sum_{k=0}^{n-1}\big(\nabla B(\bar X_x(k\ve),\xi(k))-\nabla\bar B(\bar X_x(k\ve))
   \big)Z^\ve(k\ve).
   \]
   By (\ref{2.3}) and the discrete time Gronwall inequality (see \cite{Cla}),
   \[
   |U^\ve(n\ve)|\leq e^{LT}(|V^\ve(n\ve)|+|R^\ve_2(n\ve)|).
   \]
   Hence, we obtain
   \begin{eqnarray}\label{3.22}
   &|R^\ve_1(n\ve)|\leq\ve|\sum_{k=0}^{n-1}(\nabla B(\bar X_x(k\ve),\xi(k))-\nabla\bar
   B(\bar X_x(k\ve)))(U^\ve(k\ve)\\
   &+Z^\ve(k\ve))|\leq 2L\ve\sum_{k=0}^{n-1}|U^\ve(k\ve)|+\ve\sum_{k=0}^{n-1}|V^\ve(k\ve)|\nonumber\\
   &\leq 2L\ve(e^{LT}+1)\sum_{k=0}^{n-1}|V^\ve(k\ve)|+2L\ve e^{LT}\sum_{k=0}^{n-1}|R_2^\ve(k\ve)|.
   \nonumber\end{eqnarray}

   It remains to estimate $|V^\ve(k\ve)|$ since the estimate (\ref{3.15}) of $|R_2^\ve(k\ve)|$ was
   already established above. Set
   \begin{eqnarray*}
   &\eta(k)=\nabla B(\bar X_x(k\ve),\xi(k))-\nabla\bar B(\bar X_x(k\ve))\quad\mbox{and}\\
   &\zeta_k(l)=K^\ve(k,l)\big(B(\bar X_x(l\ve),\xi(l))-\bar B(X_x(l\ve))\big).
   \end{eqnarray*}
   Then, by (\ref{3.21}) and the definition of $V^\ve$,
   \[
   V^\ve(n\ve)=\ve^2J_1(n)+\ve^2J_2(n)+\ve^2J_3(n)
   \]
   where
   \begin{eqnarray*}
   &J_1(n)=(\sum_{k=0}^{n-1}\eta(k))(B(x,\xi(0))-\bar B(x)),\\
   &J_2(n)=\sum_{k=0}^{n-1}\eta(k)
   \sum_{l=0}^{k-1}\zeta_k(l)\,\,\mbox{and}\,\, J_3(n)=\sum_{k=0}^{n-1}\eta(k)\zeta_k(k).
   \end{eqnarray*}
   By (\ref{2.3}),
   \[
   |J_1(n)|\leq 4L^2n\leq 4L^2T\ve^{-1}.
   \]

Next, we estimate $|J_2(n)|$ by Lemma \ref{lem3.4} taking into account that by (\ref{2.3})
   for $k\leq T/\ve$,
   \[
   | K^\ve(k,l)|\leq (1+Ld\ve)^{k-l}\leq (1+Ld\ve)^{T/\ve}\leq e^{LTd}
   \]
   where we take the matrix norm $\|(\al_{ij})\|=\max_j\sum_i|\al_{ij}|$. Hence,
   \[
   |\eta(k)|\leq 2L\,\,\,\mbox{and}\,\,\, |\zeta_k(l)|\leq 2Le^{LTd}.
   \]
   Since $E\eta(k)=E\zeta_k(l)=0$ and the approximation by conditional expectations condition of
    Lemma \ref{lem3.4} is satisfied with the coefficient $e^{LTd}\rho(K,n)$ we can apply Lemma \ref{3.4} to obtain that
   \[
   E\max_{0\leq n\leq T/\ve}|J_2(n)|^{2M}\leq\tilde C_1(M)\ve^{-2M}
   \]
   where $\tilde C_1(M)>0$ depends only on $M,L,T,d$ and $\rho$ but it does not depend on $\ve$.
   Here, we rely also on the assertion of Lemma \ref{lem3.4} that the constant in (\ref{3.5})
   does not depend on the sequences $\eta$ and $\zeta$ themselves since in our case these
   sequences depend on $\ve$. Finally, by the above estimates of $\eta$ and $\zeta$,
   \[
   |J_3(n)|\leq 4L^2e^{LTd}n\leq 4L^2Te^{LTd}\ve^{-1}.
   \]
   It follows that
   \[
   E\max_{0\leq n\leq T/\ve}|V^\ve(n\ve)|^{2M}\leq\ve^{2M}3^{2M-1}(4^{2M}L^{4M}T^{2M}(1+e^{LTd})
   +\tilde C_1(M)).
   \]
   This together with (\ref{3.15}) and (\ref{3.22}) yields (\ref{3.14}) completing the proof of (i).

   (ii) By (\ref{3.16}) and the standard uniform martingale moment estimates of stochastic integrals
   (see, for instance, \cite{IW} and \cite{Mao}),
   \begin{equation}\label{3.23}
   E\sup_{0\leq t\leq T}|R_4^\ve(t)|^{2M}\leq\ve^M2^{2M}M^{3M}(2M-1)^{-M}T^{M-1}C_\sig^{2M}
   \int_0^TE|H^\ve_x(s)-\bar X_x(s)|^{2M}ds.
   \end{equation}
   By (\ref{1.3}), (\ref{1.5}) and (\ref{2.3}),
   \[
   |H^\ve_x(s)-\bar X_x(s)|\leq L\int_0^s|H_x^\ve(u)-\bar X_x(u)|du +\sqrt\ve\sup_{0\leq u\leq s}
   |\int_0^u\sig(H_x^\ve(v))dW(v)|,
   \]
   and so by Gronwall's inequality,
   \begin{equation}\label{3.24}
   |H^\ve_x(s)-\bar X_x(s)|\leq e^{Ls}\sqrt\ve\sup_{0\leq u\leq s}
   |\int_0^u\sig(H^\ve_x(v))dW(v)|.
   \end{equation}
   Taking $2M$-th power of both sides of this inequality, applying the expectation, using the
   martingale moment inequalities as above and substituting the result into the right hand side
   of (\ref{3.23}), we arrive at (\ref{3.17}).

   Finally, by (\ref{2.3}) and the Cauchy-Schwarz inequality,
   \begin{eqnarray*}
   &E\sup_{0\leq t\leq T}|R^\ve_5(t)|^{2M}\leq (Ld)^{2M}E(\int_0^T|H^\ve_x(s)-\bar X_x(s)|^2ds)^{2M}\\
   &\leq (Ld)^{2M}T^{2M-1}\int_0^TE|H^\ve_x(s)-\bar X_x(s)|^{4M}ds.
   \end{eqnarray*}
   Relying on (\ref{3.24}) together with the martingale moment estimates of stochastic integrals
   we derive (\ref{3.18}) completing the proof of the lemma.
    \end{proof}

    \subsection{Characteristic functions estimates}\label{subsec3.3}

 For any $0\leq s<t\leq T$ and $\ve>0$ introduce the characteristic function
  \begin{eqnarray*}
  &f^\ve_{s,t}(w)=f^\ve_{s,t}(x,w)=E\exp(i\langle w,\, \ve^{1/2}(S^\ve(t)-S^\ve(s))\rangle)\\
  &=E\exp(i\langle w,\, \ve^{1/2}\sum_{[s/\ve]\leq k<[t/\ve]}\hat B(\bar X_x(k\ve),\xi(k))\rangle),
  \,\, w\in\bbR^d
  \end{eqnarray*}
  where $\langle\cdot,\cdot\rangle$ denotes the inner product and, recall,
  $\hat B(y,\xi(k))=B(y,\xi(k))-\bar B(y)$. The first step in the strong approximations machinery
  is an estimate of the corresponding characteristic function which in our situation amounts to
   the following.
  \begin{lemma}\label{lem3.8}
  For any $0\leq s<t\leq T$, $\ve>0$ and $x\in\bbR^d$,
  \begin{equation}\label{3.25}
  |f_{s,t}^\ve(x,w)-\exp(-\frac 12\langle(\int_s^t A(\bar X_x(u))du)w,\, w\rangle)|\leq C_2(T)
  \ve^\wp
  \end{equation}
  for all $w\in\bbR^d$ with $|w|\leq\ve^{-\wp/2}$ where we can take $\wp\leq\frac 1{30}$
  and a constant $C_2(T)>0$ does not depend on $s,t$ and $\ve$.
  \end{lemma}
  \begin{proof} Set $n=n_\ve(s,t)=[(t-s)\ve^{-1}]$.
  The left hand side of (\ref{3.25}) does not exceed 2 and for $n<16$
  we estimate it by $2(16)^\wp n^{-\wp}$ which is at least 2. So, in what follows, we will assume that $n\geq 16$,
  so that $n^{1/4}\geq 2$.
  Set $\nu(n)=[n(n^{3/4}+n^{1/4})^{-1}]$, $q_k(n)=s/\ve+k(n^{3/4}+n^{1/4})$, $r_k(n)= q_{k-1}(n)+n^{3/4}$
   for $k=1,2,...,\nu(n)$ with $q_0(n)=s/\ve$. Next, we introduce for $k=1,...,\nu(n)$,
   \begin{eqnarray*}
   &y_k=y_k(n)=\sum_{q_{k-1}(n)\leq l<r_k(n)}\hat B(\bar X_x(l\ve),\xi(l)),\, z_k=z_k(n)\\
   &=\sum_{r_k(n)\leq l<q_k(n)}\hat B(\bar X_x(l\ve),\xi(l))\,\,\mbox{and}\,\,
   z_{\nu(n)+1}=\sum_{q_{\nu(n)}\leq l\leq t/\ve}\hat B(\bar X_x(l\ve),\xi(l)).
   \end{eqnarray*}
   Then by Lemma \ref{lem3.4},
    \begin{eqnarray}\label{3.26}
   &E|\sum_{1\leq k\leq\nu(n)+1}z_k|^2\leq 2\nu(n)\sum_{1\leq k\leq\nu(n)}E|z_k|^2+2E|z_{\nu(n)+1}|^2\\
   &\leq 2C_1(1)((\nu(n))^2n^{1/4}+n^{3/4}+n^{1/4})\leq 6C_1(1)n^{3/4}.\nonumber
   \end{eqnarray}
   This together with  the Cauchy-Schwarz inequality yields,
   \begin{eqnarray}\label{3.27}
   &|f_{s,t}^\ve(x,w)-E\exp(i\langle w,n^{-1/2}\sum_{1\leq k\leq\nu(n)}y_k\rangle)|\\
   &\leq E|\exp(i\langle w,n^{-1/2}\sum_{1\leq k\leq\nu(n)+1}z_k\rangle)-1|\leq n^{-1/2}E|\langle    w,\sum_{1\leq k\leq\nu(n)+1}z_k\rangle|\nonumber\\
   &\leq n^{-1/2}|w|E|\sum_{1\leq k\leq\nu(n)+1}z_k|\leq \sqrt {6C_1(1)}|w|n^{-1/8}\nonumber
   \end{eqnarray}
   where we use that for any real $a,b$,
   \[
   |e^{i(a+b)}-e^{ib}|=|e^{ia}-1|\leq |a|.
   \]
    We will obtain (\ref{3.25}) from (\ref{3.27}) by estimating
   \begin{equation}\label{3.28}
   |E\exp(i\sum_{1\leq k\leq\nu(n)}\eta_k)-\exp(-\frac 12\langle(\int_s^t A(\bar X_x(u))du))w,w\rangle)|
   \leq I_1+I_2
   \end{equation}
   where
   \begin{eqnarray*}
   &\eta_k=\langle w,\sqrt\ve y_k\rangle,\,\,\,
   I_1=|E\exp(i\sum_{1\leq k\leq\nu(n)}\eta_k)-\prod_{1\leq k\leq\nu(n)}Ee^{i\eta_k}|\\
   &\mbox{and}\,\,\, I_2=|\prod_{1\leq k\leq\nu(n)}Ee^{i\eta_k}-\exp(-\frac 12\langle\int_s^t
    (A(\bar X_x(u)du))w,w\rangle)|.
   \end{eqnarray*}

   First, we write
   \begin{eqnarray}\label{3.29}
   &I_1\leq\sum_{m=2}^{\nu(n)}\big(|\prod_{m+1\leq k\leq \nu(n)}Ee^{i\eta_k}|\\
   &\times|E\exp(i\sum_{1\leq k\leq m}\eta_k)- E\exp(i\sum_{1\leq k\leq m-1}
   \eta_k)Ee^{i\eta_m}|\big)\nonumber\\
   &\leq\sum_{m=2}^{\nu(n)}|E\exp(i\sum_{1\leq k\leq m}\eta_k)-E\exp(i\sum_{1\leq k\leq m-1}\eta_k)
   Ee^{i\eta_m}|\nonumber
   \end{eqnarray}
   where $\prod_{\nu(n)+1\leq k\leq \nu(n)}=1$. Next, using the approximation coefficient $\rho$ and
   the inequality $|e^{ia}-e^{ib}|\leq |a-b|$, valid for any real $a$ and $b$, we obtain
   \begin{equation}\label{3.30}
   E|e^{i\eta_m}-\exp(iE(\eta_m|\cF_{q_{m-1}(n)-n^{1/4}/3,\infty}))|\leq 2\sqrt\ve
   n^{3/4}|w|\rho(K,n^{1/4}/3)
   \end{equation}
   and
   \begin{eqnarray}\label{3.31}
   &E\big\vert\exp(i\sum_{1\leq k\leq m-1}\eta_k)-\exp(iE(\sum_{1\leq k\leq m-1}
   \eta_k|\cF_{-\infty,r_{m-1}(n)+n^{1/4}/3}))\big\vert\\
   &\leq 2\sqrt\ve n^{3/4}|w|(m-1)\rho(K,n^{1/4}/3).\nonumber
   \end{eqnarray}
   Hence, by (\ref{3.30}), (\ref{3.31}) and Lemma \ref{lem3.1},
   \begin{eqnarray}\label{3.32}
   &\big\vert E\exp(i\sum_{1\leq k\leq m}\eta_k)-E\exp(i\sum_{1\leq k\leq m-1}
   \eta_k)Ee^{i\eta_m}\big\vert\\
   &\leq\big\vert E\exp\big(iE(\sum_{1\leq k\leq m-1}\eta_k|\cF_{-\infty,r_{m-1}(n)+n^{1/4}/3})
   +iE(\eta_m|\cF_{q_m(n)-n^{1/4}/3,\infty})\big)\nonumber\\
   &-E\exp(iE(\sum_{1\leq k\leq m-1}\eta_k|\cF_{-\infty,r_{m-1}(n)+n^{1/4}/3}))\nonumber\\
   &\times E\exp(iE(\eta_m|\cF_{q_{m-1}(n)-n^{1/4}/3,\infty}))\big\vert\nonumber\\
   &+2E|\sum_{1\leq k\leq m-1}\eta_k-E(\sum_{1\leq k\leq m-1}\eta_k|\cF_{-\infty,r_{m-1}(n)+n^{1/4}/3})|\nonumber\\
   &+2E|\eta_m-E(\eta_m|\cF_{q_{m-1}(n)-n^{1/4}/3,\infty})|\nonumber\\
   &\leq\varpi_{K,2M}(n^{1/4}/3)+4\sqrt\ve n^{3/4}(m-1)|w|\rho(K,n^{1/4}/3).
   \nonumber\end{eqnarray}
   This together with (\ref{3.29}) yields that
   \begin{equation}\label{3.33}
   I_1\leq n^{1/4}(\varpi_{K,2M}(n^{1/4}/3)+8\sqrt\ve n|w|\rho(K,n^{1/4}/3)).
   \end{equation}

   In order to estimate $I_2$ we observe that
   \[
   |\prod_{1\leq j\leq l}a_j-\prod_{1\leq j\leq l}b_j|\leq\sum_{1\leq j\leq l}|a_j-b_j|
   \]
   whenever $0\leq |a_j|, |b_j|\leq 1,\, j=1,...,l$, and so
   \begin{eqnarray}\label{3.34}
   &I_2\leq\sum_{1\leq k\leq \nu(n)}|Ee^{i\eta_k}-\exp(-\frac 12\langle(\int_{q_{k-1}(n)\ve}
   ^{r_k(n)\ve} A(\bar X_x(u))du)w,w\rangle)|\\
   &\leq\frac 12\sum_{1\leq k\leq \nu(n)}|E\eta_k^2-\langle(\int_{q_{k-1}(n)\ve}
   ^{r_k(n)\ve} A(\bar X_x(u))du)w,w\rangle|\nonumber\\
   &+ \sum_{1\leq k\leq \nu(n)}\big(E|\eta_k|^3+ \frac 14|\langle(\int_{q_{k-1}(n)\ve}
   ^{r_k(n)\ve} A(\bar X_x(u))du)w,w\rangle|^2\big)\nonumber
  \end{eqnarray}
  where we use (\ref{1.2}) and that for any real $a$,
  \[
  |e^{ia}-1-ia+\frac {a^2}2|\leq |a|^3\,\,\mbox{and}\,\, |e^{-a}-1+a|\leq a^2\,\,\mbox{if}\,\, a\geq 0.
  \]

  Now,
  \begin{eqnarray*}
  &E\eta_k^2=\ve E(\sum_{i=1}^dw_i\sum_{q_{k-1}(n)\leq l<r_k(n)}\hat B_i(\bar X_x(l\ve),\xi(l)))^2\\
  &=\ve\sum_{i,j=1}^dw_iw_j\sum_{q_{k-1}(n)\leq l<r_k(n)}\sum_{q_{k-1}(n)\leq m<r_k(n)}a_{ij}
  (\bar X_x(l\ve),\bar X_x(m\ve),l,m).
  \end{eqnarray*}
  Hence, by Lemma \ref{lem3.5},
  \begin{equation}\label{3.35}
  |E\eta_k^2-\langle(\int_{q_{k-1}(n)\ve}^{r_k(n)\ve}A(\bar X_x(u))du)w,w\rangle|\leq
  C_2(T)|w|^2\ve^{1/3}.
  \end{equation}
  By Lemma \ref{lem3.4} with $\gam_K=2L$  and the H\" older inequality,
  \begin{equation}\label{3.38}
  E|\eta_k|^3\leq \ve^{3/2}|w|^3\big(E(\sum_{l=q_{k-1}(n)}^{r_k(n)}\hat B(\bar   X_x(l\ve),\xi(l)))^4\big)^{3/4}\leq C_1^{3/4}(2)\ve^{3/2}n^{9/8}|w|^3.
  \end{equation}
  By the estimate of the norm of the matrix $A$ in Lemma \ref{lem3.5},
  \begin{equation}\label{3.39}
  |\langle (\int_{q_{k-1}(n)\ve}^{r_k(n)\ve}A(\bar X_x(u))du)w,w\rangle |
  \leq\hat L|w|^2n^{3/4}\ve.
  \end{equation}
  Now, collecting (\ref{3.34})--(\ref{3.39}) we obtain that
  \begin{equation}\label{3.40}
  I_2\leq C_2(T)|w|^2n^{1/4}\ve^{1/3}+C_1^{3/4}(2)\ve^{3/2}n^{11/8}|w|^3+\frac 14\hat L^2|w|^4n^{3/2}\ve^2.
  \end{equation}
  Finally, (\ref{3.26})--(\ref{3.30}), (\ref{3.33}) and (\ref{3.40}) yield (\ref{3.25}) completing the proof.
  \end{proof}

\section{Strong approximations}\label{sec4}\setcounter{equation}{0}
  \subsection{Strong approximations theorem}\label{subsec4.1}

 We will rely on the following result which is a version of Theorem 1 in \cite{BP} with some features
 taken from Theorem 4.6 in \cite{DP} (see also Theorem 3 in \cite{MP1}).
 \begin{theorem}\label{thm4.1} Let $\{ V_k,\, k\geq 1\}$ be a sequence of random vectors with values in  $\bbR^d$ defined on some probability space $(\Om,\cF,P)$ and such that $V_k$ is measurable with
  respect to $\cG_k,\, k=1,2,...$ where $\cG_k,\, k\geq 1$ is a filtration of countably generated
  sub-$\sig$-algebras of $\cF$. Assume that the probability space is rich enough so that
 there exists on it a sequence of uniformly distributed on $[0,1]$ independent random variables
 $U_k,\, k\geq 1$ independent of $\vee_{k\geq 1}\cG_k$. Let $\fG$ be a probability distribution on  $\bbR^d$ with the characteristic function $g$. Suppose that for some nonnegative numbers
 $\nu_m,\del_m$ and $K_m\geq 10^8d$,
  \begin{equation}\label{4.1}
  E\big\vert E(\exp(i\langle w,V_k\rangle)|\cG_{k-1})-g(w)\big\vert \leq\nu_k
  \end{equation}
  for all $w\in\bbR^d$ with $|w|\leq K_k$ and
  \begin{equation}\label{4.2}
  \fG\{ x:\, |x|\geq\frac 14K_k\}<\del_k.
  \end{equation}
  Then there exists a sequence $\{ W_k,\, k\geq 1\}$ of $\bbR^d$-valued random vectors defined on
  $(\Om,\cF,P)$ with the properties

  (i) $W_k$ is $\cG_k\vee\sig\{U_k\}$-measurable for each $k\geq 1$;

  (ii) each $W_k,\, k\geq 1$ has the distribution $\fG$ and $W_k$ is independent of
  $\cG_{k-1}\vee\sig\{ U_1,...,U_{k-1}\}$, and so also of $W_1,...,W_{k-1}$;

  (iii) Let $\vr_k=16K^{-1}_k\log K_k+2\nu_k^{1/2}K_k^d+2\del_k^{1/2}$. Then
  \begin{equation}\label{4.3}
  P\{ |V_k-W_k|\geq\vr_k\}\leq\vr_k.
  \end{equation}
  In particular, the Prokhorov distance between the distributions $\cL(V_k)$ of $V_k$ and $\cL(W_k)$
  of $W_k$ does not exceed $\vr_k$.
  \end{theorem}

 \subsection{Moment estimates}\label{subsec4.2}
 Here we will consider the block-gap partition similar to Lemma \ref{lem3.8}. Set $N=N_\ve(T)=[T/\ve]$,
  $q_k=q_k(N)=k([N^{3/4}]+[N^{1/4}])$, $k=0,1,...,\nu(\ve)$, where $\nu(\ve)=[\frac
  N{([N^{3/4}]+[N^{1/4}])}]$, and $r_k=r_k(N)=q_{k-1}(N)+[N^{3/4}]$ where we consider
   the intervals $[q_{k-1}(N),r_k(N)]$ as big blocks and the intervals $[r_k(N),q_k(N)]$ as
   negligible gaps. Now we define the sum
 \[
  Q_k=\sum_{q_{k-1}\leq j<r_k}E\big(\hat B(\bar X_x(j\ve),\xi(j))|\cF_{j-\frac 13[N^{1/4}],j+\frac   13[N^{1/4}]}\big)
  \]
  and observe that
  \begin{equation}\label{4.4}
  |S^\ve(t)-\sum_{1\leq k\leq\ell_\ve(t)}Q_k|\leq |R^{(1)}(t)|
  +|R^{(2)}(t)|+|R^{(3)}(t)|
  \end{equation}
 where $\ell_\ve(t)=\max\{ k:q_k\leq t/\ve\}$,
  \begin{eqnarray*}
  & R^{(1)}(t)=\sum_{1\leq k\leq\ell_\ve(t)}\sum_{q_{k-1}\leq j<r_k}\big(\hat B(\bar X_x(j\ve),\xi(j))\\
   &-E( \hat B(\bar X_x(j\ve),\xi(j))|\cF_{j-\frac 13[N^{1/4}],j+\frac 13[N^{1/4}]})\big),
   \end{eqnarray*}
   \[
    R^{(2)}(t)=\sum_{1\leq k\leq\ell_\ve(t)}\sum_{r_k\leq j<q_k}\big(\hat B(\bar     X_x(j\ve),\xi(j))\big)
   \]
  and
  \[
  R^{(3)}(t)=\sum_{r_{\ell_\ve(t)}\leq j<[t/\ve]}\hat B(\bar X_x(j\ve),\xi(j)).
  \]

 \begin{lemma}\label{lem4.2} For all $\ve>0$ and $M\geq 1$,
 \begin{equation}\label{4.5}
 E\sup_{0\leq t\leq T}(|R^{(1)}(t)|+|R^{(2)}(t)|+|R^{(3)}(t)|)^{2M}\leq C_3(M,T)\ve^{-3M/4}
 \end{equation}
 where $C_3(M,T)>0$ does not depend on $\ve$.
 \end{lemma}
 \begin{proof}
 By (\ref{2.4}),
 \[
 |R^{(1)}(t)|\leq 2T\ve^{-1}\rho(K,\frac 13T^{1/4}\ve^{-1/4}).
 \]
 Taking into account that the sum in $R^{(2)}(t)$ contains no more than $\sqrt {T/\ve}$ terms
 we obtain from Lemma \ref{lem3.4} considered with $\gam_K=2L$ that,
 \[
 E\sup_{0\leq t\leq T}|R^{(2)}(t)|^{2M}\leq C_1(M)(T/\ve)^{M/2}.
 \]
 Since the sum in $R^{(3)}(t)$ contains no more than $(T/\ve)^{3/4}+(T/\ve)^{1/4}$ terms
 we obtain again from Lemma \ref{lem3.4} that
 \[
 E\sup_{0\leq t\leq T}|R^{(3)}(t)|^{2M}\leq C_1(M)2^M(T/\ve)^{3M/4}
 \]
 completing the proof.
 \end{proof}

Next, set $\cG_k=\cF_{-\infty,r_k+\frac 13[N^{1/4}]}$.
 The following result is a corollary of Lemmas \ref{lem3.1} and \ref{lem3.8}.
 \begin{lemma}\label{lem4.3} For any $\ve\geq 0$,
 \begin{eqnarray}\label{4.6}
 &E|E(\exp(i\langle w,\,\ve^{1/2}Q_k\rangle) |\cG_{k-1})\\
 &-\exp(-\frac 12\langle(\int_{q_{k-1}\ve}^{r_k\ve}A(\bar X_x(u))du) w,w\rangle)|\leq
 C_3(T)\ve^{\wp}\nonumber
  \end{eqnarray}
  for all $w\in\bbR^d$ with $|w|\leq\ve^{-\wp/2}$ where $C_3(T)>0$ does not depend on $\ve$.
  \end{lemma}
  \begin{proof} 
  By the definition of the coefficient $\varpi$ and the above notations,
 \[
 \| E(\exp(i\langle w,\ve^{1/2}Q_k\rangle)|\cG_{k-1})-E\exp(i\langle w,\ve^{1/2}Q_k\rangle)\|_{2M}\leq\varpi_{K,2M}(N^{1/4}/3).
 \]
  Since  $|e^{i(a+b)}-e^{ib}|\leq |a|$ we obtain from (\ref{2.4}) that
  \begin{eqnarray*}
  &|E\exp(i\langle w,\ve^{1/2}Q_k\rangle)-f^\ve_{q_{k-1}\ve,r_k\ve}(w)|\\
  &\leq\ve^{1/2}|w|\sum_{q_{k-1}\leq j<r_k}E|\hat B(\bar X_x(j\ve),\xi(j))|\cF_{j-\frac 13[N^{1/4}]})\\
  &\leq \ve^{1/2}|w|(T/\ve)^{3/4}\rho(K,\frac 14[(T/\ve)^{1/4}])
  \end{eqnarray*}
  and (\ref{4.6}) follows from (\ref{2.6}) and (\ref{3.25}).
  \end{proof}

Next, we apply Theorem \ref{thm4.1} with $V_k=\ve^{1/2}Q_k,\, \cG_k$ the same as
 in Lemma \ref{lem4.3} and
 \[
 g(w)=\exp(-\frac 12\langle(\int_{q_{k-1}\ve}^{r_k\ve}A(\bar X_x(u))du) w,w\rangle)
 \]
 so that $\fG$ is the mean zero $d$-dimensional Gaussian distribution with the covariance matrix  $\int_{q_{k-1}\ve}^{r_k\ve}A(\bar X_x(u))du$. Relying on
 Lemmas \ref{lem3.8} and \ref{lem4.3} we take $\wp=\frac 1{30}$ and apply Theorem \ref{thm4.1} with
 \begin{equation*}
 K_k=\ve^{-\wp/4d}<\ve^{-\wp/2}\quad\mbox{and}\,\,\,\nu_k=C_3(T)\ve^{\wp}.
 \end{equation*}
 By the Chebyshev inequality we have also
 \begin{eqnarray*}
 &\fG\{ x:\, |x|\geq \frac {K_k}4\}=P\{|\Psi|\geq\frac 14\ve^{-\wp/4d}\}\\
 &\leq 4d(\int_0^T\| A(\bar X_x(u))\|du)\ve^{\wp/2d}\leq C_4\ve^{\wp/2d}
 \end{eqnarray*}
 for some $C_4>0$ which does not depend on $\ve$, where $\Psi$ is a random vector with the distribution $\fG$.

 Now Theorem \ref{thm4.1} provides us with random vectors $W_k,\, k\geq 1$ satisfying the properties (i)--(iii), in
 particular, the random vector $W_k$ has the mean zero Gaussian distribution with the covariance matrix  $(\int_{q_{k-1}\ve}^{r_k\ve}A(\bar X_x(u))du)$, it is
 independent of $W_1,...,W_{k-1}$ and the property (iii) holds true with
 \begin{equation*}
 \vr_k=4\frac \wp d\ve^{\wp/4d}\log(1/\ve)+2\sqrt {C_3(T)}\ve^{\wp/4}+2C_4^{1/2}\ve^{\wp/4d}.
 \end{equation*}
As a crucial corollary of Theorem \ref{thm4.1} we will obtain next a uniform $L^{2M}$-bound on the difference between the sums of $\ve^{-1/2}V_k$'s and of $\ve^{-1/2}W_{k}$'s. Set
 \[
 I(n)=I^\ve(n)=\ve^{-1/2}\sum_{k:\, r_k\leq n}(V_k-W_k).
 \]
 \begin{lemma}\label{lem4.4}
 For any $\ve>0$ small enough and $M\geq 1$,
 \begin{equation}\label{4.7}
 E\max_{0\leq n\leq T/\ve}|I(n)|^{2M}\leq C_4(M,T))\ve^{-M+\frac \wp{9d}}
 \end{equation}
 where $\wp=\frac 1{30}$ and $C_4(M)>0$ does not depend on $\ve$.
 \end{lemma}
 \begin{proof}
The proof of (\ref{4.7}) will rely on Lemmas \ref{lem3.2} and \ref{lem3.3}, and so we
 will have to estimate the conditional expectations appearing there taking into account that $V_k$ is $\cG_k=\cF_{-\infty,r_{k}+\frac 13[N^{1/4}]}$-measurable and $W_{k}$ is
 $\cG_k\vee\sig\{ U_1,...,U_k\}$-measurable where, recall, $N=[T/\ve]$. Let $k>j\geq 1$.
 Since $W_k$ is independent of $\cG_{k-1}\vee\sig\{ U_1,...,U_{k-1}\}$ we obtain that
 \begin{equation}\label{4.8}
 E(W_k|\cG_j\vee\sig\{ U_1,...,U_j\})=EW_k=0.
 \end{equation}
 Next, since $V_k$ is independent of
 $\sig\{ U_1,...,U_j\}$ and the latter $\sig$-algebra is independent of $\cG_j$ we obtain that (see, for instance, \cite{Chu}, p. 323),
 \begin{equation}\label{4.9}
 E(V_k|\cG_j\vee\sig\{ U_1,...,U_j\})=E(V_k|\cG_{j}).
 \end{equation}
 By Lemma \ref{lem3.1},
 \begin{eqnarray}\label{4.10}
 &\|E(V_k|\cG_{j})\|_{2M}\\
 &=\sqrt\ve\|\sum_{i=q_{k-1}}^{r_{k}-1} E\big( E(\hat B(\bar X_x(i\ve),\xi(i))|\cF_{i-\frac 13[N^{1/4}],i+\frac 13[N^{1/4}]})|\cF_{-\infty,r_{j}+\frac 13[N^{1/4}]})\|_{2M}\nonumber\\
 &\leq 2L\ve^{1/2}\sum_{i=q_{k-1}}^{r_{k}-1}\varpi_{K,2M}(i-r_{j}-\frac 23N^{1/4})\leq\ve^{1/2}\sum_{l=0}^{r_k-1}\varpi_{K,2M}(l).\nonumber
 \end{eqnarray}

 Now, in order to bound $A_{2M}$ from Lemma \ref{lem3.2} it remains to consider the
 case $k=n$, i.e. to estimate $\| V_k-W_k\|_{2M}$ and
 then to combine it with (\ref{4.8})--(\ref{4.10}). By (\ref{4.3}) and
 the Cauchy--Schwarz inequality for any $n\geq 1$,
 \begin{eqnarray}\label{4.11}
 &E|V_k-W_k|^{2M}=E(|V_k-W_k|^{2M}\bbI_{|V_k-W_k|\leq\vr_k})\\
 &+E(|V_k-W_k|^{2M}\bbI_{|V_k-W_k|>\vr_k})\nonumber\\
 &\leq\vr^{2M}_k+(E|V_k-W_k|^{4M})^{1/2}(P\{|V_k-W_k|>\vr_k\}^{\frac {1}2}\nonumber\\
 &\leq\vr^{2M}_k+\vr^{\frac {1}2}_k2^{2M}((E|V_k|^{4M})^{1/2}+(E|W_k|^{4M})^{1/2}).\nonumber
 \end{eqnarray}
 By Lemmas \ref{lem3.1} and \ref{lem3.4},
\begin{equation}\label{4.12}
(E|V_k|^{4M})^{1/2}\leq\sqrt {C_1(2M)}\ve^M(r_k-q_{k-1})^M\leq \sqrt {C_1(2M)}\ve^{M/4}.
\end{equation}
Since $W_k$ is a mean zero $d$-dimensional Gaussian random vector with the covariance matrix
$\int_{q_{k-1}\ve}^{r_k\ve}A(\bar X_x(u))du$ having the same distribution as the stochastic
integral $\int_{q_{k-1}\ve}^{r_k\ve}\sig(\bar X_x(u))dW(u)$, we obtain that
\begin{equation}\label{4.13}
(E|W_k|^{4M})^{1/2}\leq (2M(4M-1))^M\hat L^{M}T^{3M/4}\ve^{M/4}.
\end{equation}
Finally, taking into account that the sum of $I^\ve(n)$ contains at most $(T/\ve)^{1/4}$ summands and
combining (\ref{4.8})--(\ref{4.13}) with Lemmas \ref{lem3.2} and \ref{lem3.3}
we derive (\ref{4.7}) completing the proof (cf. Lemma 4.4 in \cite{Ki24}).
\end{proof}

 Next, let $W(t),\, t\geq 0$ be a standard $d$-dimensional Brownian motion and let $\Xi(t)$ be
 the Gaussian process given by the stochastic integral
 \[
 \Xi(t)=\int_0^t\sig(\bar X_x(u))dW(u)\,\,\mbox{where}\,\,\sig^2(x)=A(x).
 \]
 Then $\Xi(t)-\Xi(s)$ has the covariance matrix $\int_{s}^{t}A(\bar X_x(u))du$, and so
 the sequences of independent random vectors $\tilde W_k=\Xi(r_k\ve)-\Xi(q_{k-1}\ve),\, k\leq  \ell_\ve$
  and $W_k,\, k\leq\ell_\ve$ have the same distributions.
 Let $\cQ$ and $\cR$ be the joint distributions of the sequences of pairs $(V_k,W_k),\, 1\leq k\leq\ell_\ve$ and of
 $(\tilde W_k,W),\, 1\leq k\leq\ell_\ve$, respectively. Since the marginal of $\cQ$ corresponding to the sequence
 $W_k,\, 1\leq k\leq\ell_\ve$ coincides with the marginal of $\cR$ corresponding to the sequence
 $\tilde W_k,\, 1\leq k\leq\ell_\ve$ we conclude by Lemma A1 from
  \cite{BP} that we can redefine the process $\xi(n),\,n\in\bbZ$ preserving its distributions on a richer
  probability space where there exists a standard $d$-dimensional Brownian motion $W(t),\, t\in[0,T]$ and
  a sequence of random vectors $\hat W_k,\, 1\leq k\leq\ell_\ve$ such that the sequence of pairs $(V_k,\hat W_k),\, 1\leq k\leq\ell_\ve$
  and of $(\hat W_k,W),\, 1\leq k\leq\ell_\ve$ have the joint distributions $\cQ$ and $\cR$, respectively, where $V_k$'s
  are constructed by the redefined process $\xi$. Now define again $\tilde W_k=\int_{q_{k-1}\ve}^{r_k\ve}
 \sig(\bar X_x(u))dW(u)$ and let $\cH_k$ be the $\sig$-algebra generated by $\{ W(u),\, q_{k-1}\ve\leq u
 \leq r_k\ve\}$. Since $(\tilde W_k,W)$ and $(\hat W_k,W)$ have the same joint distributions, we obtain that
 \[
 E(\hat W_k|\cH_k)=E(\tilde W_k|\cH_k)=\tilde W_k\quad\mbox{a.s.}
 \]
 Then
 \begin{eqnarray*}
 &E|\hat W_k-\tilde W_k|^2=2E|\tilde W_k|^2-2E\langle\hat W_k,\tilde W_k\rangle\\
 &=2E|\tilde W_k|^2-2E\langle\tilde W_k,E(\hat W_k|\cH_k)\rangle=0
 \end{eqnarray*}
 where $\langle\cdot,\cdot\rangle$ is the inner product. Hence,
 \[
 \hat W_k=\tilde W_k=\Xi(r_k\ve)-\Xi(q_{k-1}\ve)=\int_{q_{k-1}\ve}^{r_k\ve} \sig(\bar X_x(u))dW(u)\,\,\,\mbox{a.s.}
 \]
 From now on we drop the hat and tilde signs over $W_k$ and claim in view of the above that $W_k=\Xi(r_k\ve)-\Xi(q_{k-1}\ve)$,
 $1\leq k\leq\ell_\ve$ satisfy (\ref{4.3}) and (\ref{4.7}).


 Now, Lemmas \ref{lem4.2} and \ref{lem4.4} yield that
\begin{eqnarray}\label{4.14}
&\quad E\sup_{0\leq t\leq T}|\sqrt\ve S^\ve(t)-\int_{0}^{t}\sig(\bar X_x(u))dW(u)|^{2M}
\leq 4^{2M-1}\big(C_3(M,T)\ve^{M/4}\\
&+C_4(M,T))\ve^{\frac M4(1+\frac \wp{2d})}
+E\sup_{0\leq t\leq T}|J_1(t)|^{2M}+E\sup_{0\leq t\leq T}|J_2(t)|^{2M}\big)\nonumber
\end{eqnarray}
where
\[
J_1(t)=\sum_{1\leq k\leq\ell_\ve(t)}\int_{r_{k}\ve}^{q_k\ve}\sig(\bar X_x(u))dW(u)
\]
and
\begin{eqnarray*}
&J_2(t)=\int_0^t\sig(\bar X_x(u))dW(u)-\sum_{1\leq k\leq\ell_\ve(t)}\int_{q_{k-1}\ve}^{q_k\ve}
\sig(\bar X_x(u))dW(u)\\
&=\int_{q_{\ell_\ve(t)}\ve}^t\sig(\bar X_x(u))dW(u).
\end{eqnarray*}
By the standard martingale estimates of stochastic integrals (see, for instance, \cite{IW} and \cite{Mao}),
\[
E\sup_{0\leq t\leq T}|J_1(t)|^{2M}\leq C_5(M,T)\ve^{M/2}\,\,\mbox{and}\,\,
E\sup_{0\leq t\leq T}|J_2(t)|^{2M}\leq C_5(M,T)\ve^{M/4}
\]
where $C_5(M,T)>0$ does not depend on $\ve$.
These together with Lemmas \ref{lem3.6}, \ref{lem3.7}, \ref{lem4.2} and the estimate (\ref{4.14})
complete the proof of Theorem \ref{thm2.1}.   \qed

 \subsection{Proof of Corollary \ref{cor2.2}}\label{subsec4.3}
 For Corollary \ref{cor2.2} observe that if $X$ and $Y$ are two random variables on a metric
 space $\cX$ with a metric $d$ then for any $\gam>0$ and a set $U\subset\cX$,
 \[
 \{ X\in U\}\subset\{ Y\in U^\gam\}\cup\{ d(X,Y)\geq\gam\}.
 \]
 Hence, by the Chebyshev inequality
 \[
 P\{ X\in U\}\leq P\{ Y\in U^\gam\}+P\{ d(X,y)\geq\gam\}\leq P\{ Y\in U^\gam\} +q_{2M}\gam^{-2M}
 \]
 provided $E(d(X,Y))^{2M}\leq q_{2M}$. Similarly, $P\{ Y\in U\}\leq P\{ X\in U^\gam\}
 +q_{2M}\gam^{-2M}$, and so
 \[
 \pi(\cL(X),\cL(Y))\leq\max(\gam,q_{2M}\gam^{-2M}).
 \]
 Taking $X=\ve^{-1/2}(X_x^\ve-\bar X_x)$ or $X=\ve^{-1/2}(H_x^\ve-\bar X_x)$ and $Y=G$
 we derive (\ref{2.14}) from (\ref{2.11}) and (\ref{2.12}) by choosing either $q_{2M}=C_0(M)\ve^\del,
 \,\gam=C_0^{1/3}(M)\ve^{\del/3},\, M=1$ or $q_{2M}=\hat C^{(2M+1)/3}_0(M)\ve^{M(2M+1)/3},\,\gam=\hat C_0^{1/3}(M)\ve^{M/3}$.

 \qed


\section{Continuous time case}\label{sec5}\setcounter{equation}{0}
\subsection{Discretization}\label{subsec5.1}
Introduce the discrete time process $y^\ve$ by the recurrence relation
\[
y^\ve_x((k+1)\ve,\om)=y^\ve_x(k\ve,\om)+\ve b(y^\ve_x(k\ve,\om),\vt^k\om),\,\,\, y^\ve_x(0)=x
\]
and $y^\ve_x(t,\om)=y^\ve_x(k\ve,\om)$ if $k\ve\leq t<(k+1)\ve$. The following result can be derived easily
from Lemma 3.1 in \cite{Ki95} but we will give its independent proof here for completeness.
\begin{lemma}\label{lem5.1}
For any $\ve>0$,
\begin{eqnarray}\label{5.1}
&\sup_{x\in\bbR^d}\sup_{t\in[0,T]}\sup_{\om\in\Om}\sup_{0\leq s\leq\tau(\om)}|X^\ve_x(t,(\om,s))-y^\ve_x(\ve n(t/\ve,\om),\om)|\\
&\leq\ve L\bar L(2+e^{LT}+\bar Le^{L\bar LT})\nonumber
\end{eqnarray}
where for each $t\geq 0$ we set
\[
n(t,\om)=\max\{ k\geq 0:\,\sum_{j=0}^{k-1}\tau\circ\vt^j(\om)\leq t\}.
\]
\end{lemma}
\begin{proof}
First, we write for any $t\leq T/\ve$ and $s\leq\tau(\om)$ that
\begin{eqnarray*}
&X^\ve_x(\ve t,(\om,s))=x+\ve\int_0^tB(X^\ve_x(\ve u,(\om,s)),\xi(u,(\om,s)))du\\
&=x+\ve\int_0^tB(X^\ve_x(\ve u,(\om,s)),\xi(u+s,(\om,0)))du\\
&=x+\ve\int_0^{t+s}B(X^\ve_x(\ve(v-s),(\om,s)),\xi(v,(\om,0)))dv.
\end{eqnarray*}
Set $X^\ve_{x,s}(\ve v,(\om,0))=X^\ve_x(\ve(v-s),(\om,s))$ for $v\geq s$, so that
$X^\ve_{x,s}(\ve s,(\om,0))=x$. Then
\[
X^\ve_{x,s}(\ve(t+s),(\om,0))=x+\ve\int_s^{s+t}B(X^\ve_{x,s}(\ve v,(\om,0)),\xi(v,(\om,0)))dv.
\]
Hence, by (\ref{2.3}),
\begin{eqnarray}\label{5.2}
&|X^\ve_x(\ve(t+s),(\om,0))-X^\ve_{x,s}(\ve(t+s),(\om,0))|\\
&\leq |X^\ve_x(\ve s,(\om,0))-x|+
\ve L\int_s^{s+t}|X^\ve_x(\ve v,(\om,0))-X^\ve_{x,s}(\ve v,(\om,0))|dv.\nonumber
\end{eqnarray}
Since by (\ref{1.1}), (\ref{2.3}) and (\ref{2.15}),
\[
|X^\ve_x(\ve s,(\om,0))-x|\leq\ve L\bar L\,\,\mbox{and}\,\,|X^\ve_x(\ve(t+s),(\om,0))-X^\ve_{x}(\ve t,(\om,0))|\leq
\ve L\bar L,
\]
we obtain from (\ref{5.2}) by the Gronwall inequality that for any $t\leq T/\ve$ and $s\leq\tau(\om)$,
\begin{eqnarray}\label{5.3}
&|X^\ve_x(\ve t,(\om,0))-X^\ve_{x}(\ve t,(\om,s))|\leq |X^\ve_x(\ve t,(\om,0))-X^\ve_{x,s}(\ve(t+s),(\om,0))|\\
&+|X^\ve_{x,s}(\ve(t+s),(\om,0))-X^\ve_x(\ve(t+s),(\om,0))|\nonumber\\
&+|X^\ve_x(\ve(t+s),(\om,0))-X^\ve_{x}(\ve t,(\om,0))|\leq\ve L\bar L(2+e^{LT}).\nonumber
\end{eqnarray}
It follows that in order to prove (\ref{5.1}) it suffices to consider there just $X_x^\ve(t,(\om,0))$ which
we denote by $X^\ve_x(t,\om)$ and take $t\in[0,T]$.

Set
\[
\Te_k(\om)=\sum_{j=0}^{k-1}\tau\circ\vt^j(\om),\,\,\Te_0(\om)=0.
\]
By (\ref{2.3}) and (\ref{2.15}) for any $\Te_n(\om)\leq t<\Te_{n+1}(\om)$,
\begin{equation}\label{5.4}
|X^\ve_x(\ve t,\om)-X^\ve_x(\ve\Te_n(\om),\om)|\leq\ve\int_{\Te_n(\om)}^t|B(X_x^\ve(\ve u,\om),\xi(u,\om))|du
\leq\ve L\bar L.
\end{equation}
Next, by (\ref{2.3}), (\ref{2.15}) and (\ref{5.4}) for any $n\leq T/\ve$,
\begin{eqnarray*}
&|X_x^\ve(\ve\Te_n(\om),\om)-y^\ve_x(\ve n,\om)|\\
&\leq\ve\sum_{k=0}^{n-1}|\int_{\Te_k(\om)}^{\Te_{k+1}(\om)}B(X^\ve_x(\ve u,\om),
\xi(u,\om))du-b(y_x^\ve(\ve k,\om),\vt^k\om)|\\
&\leq\ve\sum_{k=0}^{n-1}|\int_{\Te_k(\om)}^{\Te_{k+1}(\om)}B(X^\ve_x(\ve\Te_k(\om),\om),\xi(u,\om))du-b(y_x^\ve(\ve k,\om),\vt^k\om)|\\
&=\ve\sum_{k=0}^{n-1}|b(X^\ve_x(\ve\Te_k(\om),\om),\vt^k\om)-b(y_x^\ve(\ve k,\om),\vt^k\om)|\\
&\leq\ve L\bar L\sum_{k=0}^{n-1}|X^\ve_x(\ve\Te_k(\om),\om)-y_x^\ve(\ve k,\om)|+\ve L\bar LT.
\end{eqnarray*}
By the discrete time Gronwall inequality (see \cite{Cla}) we obtain from here that for any $n\leq T/\ve$,
\[
|X^\ve_x(\ve\Te_n(\om),\om)-y_x^\ve(\ve n,\om)|\le\ve L\bar L^2Te^{L\bar LT}
\]
which together with (\ref{5.3}) and (\ref{5.4}) yields (\ref{5.1}).
\end{proof}

Next, set $g(x,\om)=\tau(\om)\bar B(x)$ and introduce the discrete time process $z^\ve_x$ by the recurrence relation
\[
z^\ve_x((k+1)\ve,\om)=z^\ve_x(k\ve,\om)+\ve g(z^\ve_x(k\ve,\om),\te^k\om),\,\, z^\ve_x(0)=x
\]
and $z^\ve_x(t,\om)=z^\ve_x(k\ve,\om)$ if $k\ve\leq t<(k+1)\ve$. Observe that $g(x,\om)$ is obtained in the same way as $b(x,\om)$
when we replace $B(x,\xi(s,\om))$ by $\bar B(x)$. Hence, we can apply Lemma \ref{lem5.1} to the pair $\bar X_x$ and $z^\ve_x$ in place
 of the pair $X^\ve_x$ and $y^\ve_x$ and looking carefully at the proof there we see that this lemma can be applied here with the same
 constants. Thus, for all $\ve>0$,
 \begin{equation}\label{5.5}
\sup_{x\in\bbR^d}\sup_{t\in[0,T]}\sup_{\om\in\Om}|\bar X_x(t)-z^\ve_x(\ve n(t/\ve,\om),\om)|
\leq\ve L\bar L(2+e^{LT}+\bar Le^{L\bar LT}).
\end{equation}
Now observe that
\begin{equation}\label{5.6}
\bar b(x)=Eb(x,\om)=\bar\tau\bar B(x)=Eg(x,\om)=\bar g(x).
\end{equation}
Hence,
\[
\bar y_x(t)=x+\int_0^t\bar b(\bar y_x(s)ds=x+\int_0^t\bar g(\bar y_x(s))ds=x+\bar\tau\int_0^t\bar B(\bar y_x(s))ds.
\]
Since
\[
\bar z_x(t)=x+\int_0^t\bar g(\bar z_x(s))ds\,\,\mbox{and}\,\,\bar X_x(\bar\tau t)=x+\int_0^{\bar\tau t}\bar B(\bar X_x(s))ds=x+\bar\tau\int_0^t\bar B(\bar X_x(\bar\tau u))du,
\]
we conclude by uniqueness of the solutions of the equations above that
\begin{equation}\label{5.7}
\bar y_x(t)=\bar z_x(t)=\bar X_x(\bar\tau t)\,\,\,\mbox{for all}\,\, t\in[0,T].
\end{equation}
It follows from (\ref{5.1}), (\ref{5.5}) and (\ref{5.7}) that for all $\ve>0$,
\begin{eqnarray}\label{5.8}
&\sup_{x\in\bbR^d}\sup_{t\in[0,T]}\sup_{\om\in\Om}\sup_{0\leq s\leq\tau(\om)}|(X^\ve_x(t,(\om,s))-\bar X_x(t))\\
&-(y^\ve_x(\ve n(t/\ve,\om),\om)-\bar y_x(\ve n(t/\ve,\om)))
+(z^\ve_x(\ve n(t/\ve,\om),\om)-\bar z_x(\ve n(t/\ve,\om)))|\nonumber\\
&\leq\ve L\bar L(2+e^{LT}+\bar Le^{L\bar LT}).\nonumber
\end{eqnarray}

\subsection{Time change estimates}
Next, we compare $y^\ve_x(\ve n(t/\ve,\om),\om)-\bar y_x(\ve n(t/\ve,\om))$ and $z^\ve_x(\ve n(t/\ve,\om),\om)-\bar z_x(\ve n(t/\ve,\om))$
with $y^\ve_x(t/\bar\tau,\om)-\bar y_x(t/\bar\tau)$ and  $z^\ve_x(t/\bar\tau,\om)-\bar z_x(t/\bar\tau)$, respectively.
\begin{lemma}\label{lem5.2} For all $\ve>0$,
\begin{eqnarray}\label{5.9}
&E\sup_{0\leq t\leq T}| y^\ve_x(\ve n(t/\ve,\om),\om)-\bar y_x(\ve n(t/\ve,\om))\\
&-(y^\ve_x(t/\bar\tau,\om)-\bar y_x(t/\bar\tau))|^{2M}\leq C_6(M)\ve^{(3M-4)/2}\nonumber
\end{eqnarray}
and
\begin{eqnarray}\label{5.10}
&E\sup_{0\leq t\leq T}| z^\ve_x(\ve n(t/\ve,\om),\om)-\bar z_x(\ve n(t/\ve,\om))\\
&-(z^\ve_x(t/\bar\tau,\om)-\bar z_x(t/\bar\tau))|^{2M}\leq C_6(M)\ve^{(3M-4)/2}\nonumber
\end{eqnarray}
where $C_6(M)>0$ does not depend on $\ve$.
\end{lemma}
\begin{proof}
In the proof we will rely on Lemma \ref{5.1} from \cite{FK} saying that for any $M\geq 1$ and $\ve>0$,
\begin{eqnarray}\label{5.11}
&E|n(t/\ve,\om)-t/\ve\bar\tau|^{2M}\leq K(M)(t/\ve\bar\tau)^M\,\,\,\mbox{and}\\
&E\sup_{0\leq s\leq t}|n(s/\ve,\om)-s/\ve\bar\tau|^{2M}\leq K(M)(t/\ve\bar\tau)^{M+1}\nonumber
\end{eqnarray}
where $K(M)>0$ does not depend on $\ve$ and $t$. As at the beginning of the proof of Lemma \ref{lem3.7} it will be convenient
to replace $\bar y_x$ and $\bar z_x$ by $\hat y^\ve_x$ and $\hat z_x^\ve$ given by
\begin{eqnarray*}
&\hat y^\ve_x(n\ve)=x+\ve\sum_{k=0}^{n-1}\bar b(\bar y_x(k\ve))=x+\int_0^{n\ve}\bar b(\bar y_x([s/\ve]\ve))ds\,\,\,\mbox{and}\\
&\hat z^\ve_x(n\ve)=x+\ve\sum_{k=0}^{n-1}\bar g(\bar z_x(k\ve))=x+\int_0^{n\ve}\bar g(\bar z_x([s/\ve]\ve))ds
\end{eqnarray*}
with $\hat y^\ve_x(t)=\hat y^\ve_x(k\ve)$ and $\hat z^\ve_x(t)=\hat z^\ve_x(k\ve)$ if $k\ve\leq t<(k+1)\ve$. Then, in the same way as in
Lemma \ref{lem3.7} for all $n\leq N$,
\[
|\bar y_x(n\ve)-\hat y^\ve_x(n\ve)|\leq L^2N\ve^2\,\,\mbox{and}\,\,|\bar z_x(n\ve)-\hat z^\ve_x(n\ve)|\leq L^2N\ve^2,
\]
and so for all $0\leq t\leq\tilde T=\bar LT$,
\begin{equation}\label{5.12}
|\bar y_x(t)-\hat y^\ve_x(t)|\leq L^2\tilde T\ve\,\,\mbox{and}\,\,|\bar z_x(t)-\hat z^\ve_x(t)|\leq L^2N\ve^2.
\end{equation}
Here we take $\tilde T$ and not just $T$ since the time $\ve n(t/\ve,\om)$ can run, in principle, up to $\bar LT\geq T$.
The estimate (\ref{5.12}) eables us to replace $\bar y_x$ and $\bar z_x$ by $\hat y_x$ and $\hat z_x$, respectively, and so
we estimate now
\begin{eqnarray}\label{5.13}
&|y^\ve_x(\ve n(t/\ve,\om),\om)-\hat y_x(\ve n(t/\ve,\om))-(y^\ve_x(t/\bar\tau,\om)-\hat y_x(t/\bar\tau))|\\
&=\ve|\sum_{\min(t/\ve\bar\tau,n(t/\ve,\om))\leq k<\max(t/\ve\bar\tau,n(t/\ve,\om))}(b(y^\ve_x(k\ve,\om),\vt^k\om)-\bar b(\bar y_x(k\ve)))|\nonumber\\
&\leq\ve\cI^\ve_x(t,\om)+\ve\cJ_x^\ve(t,\om)\nonumber
\end{eqnarray}
where by (\ref{2.3}) and (\ref{2.15}),
\begin{eqnarray}\label{5.14}
&\cI_x^\ve(t,\om)=|\sum_{\min(t/\ve\bar\tau,n(t/\ve,\om))\leq k<\max(t/\ve\bar\tau,n(t/\ve,\om))}(b(y^\ve_x(k\ve,\om),\vt^k\om)\\
&-b(\bar y_x(k\ve),\vt^k\om))|
 \leq L\bar L|n(t/\ve,\om)-t/\ve\bar\tau|\sup_{0\leq t\leq\bar LT}|y^\ve_x(t,\om)-\bar y_x(t)|\,\,\,\mbox{and}\nonumber\\
&\cJ_x(t,\om)=|\sum_{\min(t/\ve\bar\tau,n(t/\ve,\om))\leq k<\max(t/\ve\bar\tau,n(t/\ve,\om))}(b(\bar y^\ve_x(k\ve),\vt^k\om)\nonumber\\
&-Eb(\bar y_x(k\ve),\vt^k\om))|.\nonumber
\end{eqnarray}
 By (\ref{3.19}) applied to $y^\ve_x$ and $\bar y_x$ in place of $X^\ve_x$ and $\bar X_x$ together with (\ref{5.11}) (with
$M$ replaced by $2M$) and the Cauchy--Schwarz inequality we obtain
\begin{eqnarray}\label{5.15}
&\quad E\sup_{0\leq t\leq T}|\cI^\ve_x(t,\om)|^{2M}\leq (L\bar L)^{2M}(E\sup_{0\leq t\leq T}|n(t/\ve,\om)-t/\ve\bar\tau|^{4M})^{1/2}\\
&\times ( E\sup_{0\leq t\leq \bar LT}|y^\ve_x(t,\om)-\bar y_x(t)|^{4M})^{1/2}\leq C_{7}(M,T)\ve^{1/2}\nonumber
\end{eqnarray}
where $C_{7}(M,T)>0$ does not depend on $\ve$.

In order to estimate the second term in the right hand side of (\ref{5.13}) introduce for $j=1,2,...$ the events
\[
A_j=\{\om:\,(j-1)\ve^{-1/2}\leq\sup_{0\leq t\leq T}|n(t/\ve,\om)-t/\ve\bar\tau|<j\ve^{-1/2}\}.
\]
In fact, $A_j$ is empty for any $j>\bar LT\ve^{-1/2}$ since
\[
\sup_{0\leq t\leq T}\max(n(t/\ve,\om),t/\ve\bar\tau)\leq T\bar L\ve^{-1}.
\]
It follows that
\begin{eqnarray}\label{5.16}
&\,\,\, E\sup_{0\leq t\leq T}|\cJ^\ve_x(t,\om)|^{2M}=\sum_{1\leq j\leq \bar LT\ve^{-1/2}}E\bbI_{A_j}\sup_{0\leq t\leq T}|\cJ^\ve_x(t,\om)|^{2M}\\
&\leq\sum_{1\leq j\leq \bar LT\ve^{-1/2}}E\big(\bbI_{A_j}\max_{0\leq k\leq \bar LT\ve^{-1}}\max_{1\leq n<j\ve^{-1/2}}\nonumber\\
&|\sum_{k\leq l\leq k+n}(b(\bar y^\ve_x(l\ve),\vt^l\om)-Eb(\bar y_x(l\ve),\vt^l\om))|^{2M}\big)\nonumber\\
&\leq\sum_{1\leq j\leq \bar LT\ve^{-1/2}}\sum_{0\leq k\leq \bar LT\ve^{-1}}R^\ve_x(j,k,n,M)\nonumber
\end{eqnarray}
where $\bbI_A$ is the indicator of an event $A$ and by the Cauchy--Schwarz inequality,
\begin{eqnarray}\label{5.17}
&R^\ve_x(j,k,n,M)=E\big(\bbI_{A_j}\max_{1\leq n<j\ve^{-1/2}}|\sum_{k\leq l\leq k+n}(b(\bar y^\ve_x(l\ve),\vt^l\om)\\
&-Eb(\bar y_x(l\ve),\vt^l\om))|^{2M}\big)\nonumber\\
&\leq (P(A_j))^{1/2}\big(E\big(\max_{1\leq n<j\ve^{-1/2}}|\sum_{k\leq l\leq k+n}(b(\bar y^\ve_x(l\ve),\vt^l\om)\nonumber\\
&-Eb(\bar y_x(l\ve),\vt^l\om))|^{4M}\big)\big)^{1/2}.\nonumber
\end{eqnarray}

For $j=1$ we estimate $P(A_j)$ just by $1$ and for $j\geq 2$ we apply (\ref{5.11}) and the Chebyshev inequality to obtain
\begin{eqnarray}\label{5.18}
&P(A_j)\leq P\{\sup_{0\leq t\leq T}|n(t/\ve,\om)-t/\ve\bar\tau|\geq(j-1)\ve^{-1/2}\}\\
&\leq\ve^M(j-1)^{-2M}E\sup_{0\leq t\leq T}|n(t/\ve,\om)-t/\ve\bar\tau|^{2M}\nonumber\\
&\leq K(M)\ve^{-1}(j-1)^{2M}(T/\bar\tau)^{M+1}.\nonumber
\end{eqnarray}
The second factor in the right hand side of (\ref{5.17}) we estimate by Lemma \ref{lem3.4} with $\eta_j=b(\bar y_x((k+j)\ve),\vt^{k+j}\om)
-Eb(\bar y_x((k+j)\ve),\vt^{k+j}\om)$ to derive
\begin{equation}\label{5.19}
E\big(\max_{1\leq n<j\ve^{-1/2}}|\sum_{k\leq l\leq k+n}(b(\bar y_x(l\ve),\vt^{l}\om)-Eb(\bar y_x(l\ve),\vt^{l}\om))|^{2M}\big)
\leq C_1(M)(j\ve^{-1/2}+1)^M
\end{equation}
where $C_1(M)>0$ does not depend on $\ve$ and $j$. Combining (\ref{5.16})--(\ref{5.19}) we conclude that
\begin{equation}\label{5.20}
E\sup_{0\leq t\leq T}|\cJ^\ve_x(t,\om)|^{2M}\leq C_{8}(M)\ve^{-(M+4)/2}
\end{equation}
for some $C_{8}(M)>0$ which does not depend on $\ve$. Finally, (\ref{5.12})--(\ref{5.15}) and (\ref{5.20}) yield
(\ref{5.9}) and (\ref{5.10}) completing the proof of the lemma.
\end{proof}

Observe that, as a byproduct, (\ref{3.19}) together with (\ref{5.8})--(\ref{5.10}) improves the estimate
of Theorem A in \cite{DG}.

\subsection{Completing the proof of Theorem \ref{thm2.4}}\label{subsec5.3}
In the same way as in (\ref{3.13}), using the Taylor formula we can write
\begin{equation}\label{5.21}
y^\ve_x(n\ve)-\bar y_x(n\ve)=\sum_{0\leq k<n}\nabla\bar b(\bar y_x(k\ve))(y^\ve_x(k\ve)-\bar y_x(k\ve))+\ve S^\ve_y(n\ve)
+R^\ve_{1,y}(n\ve)+R^\ve_{2,y}(n\ve)
\end{equation}
and
\begin{equation}\label{5.22}
z^\ve_x(n\ve)-\bar z_x(n\ve)=\sum_{0\leq k<n}\nabla\bar g(\bar z_x(k\ve))(z^\ve_x(k\ve)-\bar z_x(k\ve))+\ve S^\ve_z(n\ve)
+R^\ve_{1,z}(n\ve)+R^\ve_{2,z}(n\ve)
\end{equation}
where
\[
S^\ve_y(s)=\sum_{0\leq k<[s/\ve]}(b(\bar y_x(k\ve),\vt^k\om)-\bar b(\bar y_x(k\ve))),
\]
\[
S^\ve_z(s)=\sum_{0\leq k<[s/\ve]}(g(\bar z_x(k\ve),\vt^k\om)-\bar g(\bar z_x(k\ve)))
\]
and by Lemma \ref{lem3.7},
\begin{eqnarray}\label{5.23}
&E\big(\sup_{0\leq t\leq T}|R^\ve_{1,y}(t)|^{2M}+\sup_{0\leq t\leq T}|R^\ve_{2,y}(t)|^{2M}\\
&+\sup_{0\leq t\leq T}|R^\ve_{1,z}(t)|^{2M}+\sup_{0\leq t\leq T}|R^\ve_{2,z}(t)|^{2M})\leq C_{9}(M)\ve^{2M}\nonumber
\end{eqnarray}
for some $C_{9}(M)>0$ which does not depend on $\ve$.

Next, let $G$ be the Gaussian process solving the linear equation (\ref{2.16}). Then by (\ref{5.6}), (\ref{5.7}), (\ref{5.21})
and (\ref{5.22}),
\begin{eqnarray}\label{5.24}
&|y^\ve_x(t)-z^\ve_x(t)-\sqrt\ve G(t)|\leq L\bar L\int_0^t|y^\ve_x(s)-z^\ve_x(s)-\sqrt\ve G(s)|ds\\
&+|\ve(S^\ve_y(t)-S^\ve_z(t))-\sqrt\ve\int_0^t\sig(\bar X_x(\bar\tau s))dW(s)|\nonumber\\
&+|R^\ve_{1,y}(t)|+|R^\ve_{2,y}(t)| +|R^\ve_{1,z}(t)|+|R^\ve_{2,z}(t)|+|\hat R^\ve(t)|\nonumber
\end{eqnarray}
where
\begin{equation}\label{5.25}
\hat R^\ve(t)=\bar\tau\int_0^t|\nabla\bar B(\bar X_x(\bar\tau s))-\nabla\bar B(\bar X_x(\bar\tau[s/\ve]\ve))||G(s)|ds
\leq L\bar L\ve\int_0^t|G(s)|ds.
\end{equation}
By the Gronwall inequality,
\begin{eqnarray}\label{5.26}
&|\ve^{-1/2}(y_x^\ve(t)-z^\ve_x(t))-G(t)|\leq e^{L\bar LT}\big(|\sqrt\ve(S^\ve_y(t)-S^\ve_z(t))\\
&-\int_0^t\sig(\bar X_x(\bar\tau s))dW(s)|\nonumber\\
&+\ve^{-1/2}(|R^\ve_{1,y}(n\ve)|+|R^\ve_{2,y}(n\ve)| +|R^\ve_{1,z}(n\ve)|+|R^\ve_{2,z}(n\ve)|+|\hat R^\ve(t)\big)|).\nonumber
\end{eqnarray}
In order to estimate $\hat R^\ve$ we write
\[
E\sup_{0\leq t\leq T}(\int_0^t|G(s)|ds)^{2M}=E(\int_0^T|G(s)|ds)^{2M}\leq T^{2M-1}\int_0^TE|G(s)|^{2M}ds.
\]
By (\ref{2.3}), (\ref{2.15}), (\ref{2.16}) and the standard moment estimates of stochastic integrals we obtain
that for any $t\in[0,T]$,
\begin{eqnarray*}
&E|G(t)|^{2M}\leq 2^{2M-1}(L\bar L)^{2M}E(\int_0^t|G(s)|ds)^{2M}+2^{2M-1}E(\int_0^t\sig(\bar X_x(\bar\tau s))dW(s))^{2M}\\
&\leq 2^{2M-1}(L\bar L)^{2M}T^{2M-1}\int_0^tE|G(s)|^{2M}ds+2^{2M-1}(M(2M-1))^MT^M\sup_x|\sig(x)|^{2M}.
\end{eqnarray*}
Hence, by the Gronwall inequality
\begin{eqnarray*}
&E|G(t)|^{2M}\leq\tilde C(T,M)\\
&=2^{2M-1}(M(2M-1))^MT^M\sup_x|\sig(x)|^{2M}\exp\big(2^{2M-1}(L\bar L)^{2M}T^{2M-1}\big),
\end{eqnarray*}
and so by (\ref{5.25}),
\begin{equation}\label{5.27}
E\sup_{0\leq t\leq T}|\hat R^\ve(t)|\leq (L\bar L)^{2M}\ve^{2M}T^{2M-1}\tilde C(T,M).
\end{equation}

Next, in the same way as in Section \ref{sec4} we construct for each $\ve>0$ the Brownian motion $W=W_\ve$ such that
\begin{equation}\label{5.28}
E\sup_{0\leq t\leq T}|\sqrt\ve(S^\ve_y(t/\bar\tau)-S^\ve_z(t/\bar\tau))-\int_0^t\sig(\bar X(s))dW_\ve(s)|^{2M}\leq C_{10}(M)\ve^{\del}
\end{equation}
where $C_{10}(M)>0$ does not depend on $\ve$. This together with (\ref{5.8})--(\ref{5.10}), (\ref{5.23}), (\ref{5.26}) and
(\ref{5.27}) yields (\ref{2.20}) and (\ref{2.21}) while (\ref{2.19}) and (\ref{2.22}) follow from here and (\ref{3.12}) taking into account
(\ref{3.17}) and (\ref{3.18}). The result similar to Corollary \ref{cor2.2} with the same constants is obtained
in the continuous and discrete time cases in the same way. This completes the proof of Theorem \ref{thm2.4}.
\qed

\section{Almost sure approximations and the law of iterated logarithm}\label{sec6}\setcounter{equation}{0}
\subsection{Some comparisons}\label{subsec6.1}

We will start with the following result.
\begin{lemma}\label{lem6.1} For any fixed $\ka>0$ as $\ve\downarrow 0$,
\begin{equation}\label{6.1}
\sup_{0\leq t\leq T}|X^\ve_x(t)-\bar X_x(t)-Z^\ve(t)|=O(\ve^{1-\ka})\quad\mbox{a.s.}
\end{equation}
where $Z^\ve$ was defined by (\ref{3.20}).
\end{lemma}
\begin{proof} Though it will suffice for our purposes to have in the right
hand side of (\ref{6.1}) only $O(\ve^{\frac 12+\del})$ for an arbitrarily small $\del>0$,
we will prove (\ref{6.1}) in the stronger form as stated. Using (\ref{3.13}) we obtain
\begin{eqnarray}\label{6.2}
&X^\ve_x(n\ve)-\bar X_x(n\ve)-Z^\ve(n\ve)\\
&=\ve\sum_{k=0}^{n-1}\nabla\bar B(\bar X_x(k\ve))(X^\ve_x(k\ve)-\bar X_x(k\ve)-Z^\ve(k\ve))
+R^\ve_1(n\ve)+R^\ve_2(n\ve),\nonumber
\end{eqnarray}
and so by the discrete time Gronwall inequality,
\begin{equation}\label{6.3}
\max_{0\leq n\leq T/\ve}|X^\ve_x(n\ve)-\bar X_x(n\ve)-Z^\ve(n\ve)|\leq e^{LT}\sup_{0\leq t\leq T}(|R^\ve_1(t)
+|R^\ve_2(t)|+|R^\ve_2(t)|).
\end{equation}

By (\ref{3.19}) for any $\la>0$,
\[
P\{\sup_{0\leq t\leq T}|X^\ve_x(t)-\bar X_x(t)|>\ve^{\frac 12-\la}\}\leq \check C_T(M)\ve^{2M\la},
\]
and so taking $M\geq\la^{-1}$ and $\ve_n=\frac 1n$ we obtain by the Borel--Cantelli lemma that for
$n=1,2,...,$
\begin{equation}\label{6.4}
\sup_{0\leq t\leq T}|X^{\ve_n}_x(t)-\bar X_x(t)|=O(\ve_n^{\frac 12-\la})\quad\mbox{a.s.}
\end{equation}
Now, set $\Psi_n(\ve)=X^\ve_x(n\ve)$. Then by (\ref{1.6}),
\[
\Psi_{m+1}(\ve)=x+\ve\sum_{k=0}^mB(\Psi_k(\ve),\xi(k)),\,\, m+1\leq T/\ve,
\]
and so by (\ref{2.3}) for $\ve_n=\frac 1n\leq\ve<\ve_{n-1}=\frac 1{n-1}$,
\[
|\Psi_m(\ve)-\Psi_m(\ve_n)|\leq\ve L\sum_{k=0}^{m-1}|\Psi_k(\ve)-\Psi_k(\ve_n)|+(\ve-\ve_n)LT\ve^{-1}.
\]
Hence, by the discrete time Gronwall inequality for all $m\leq T/\ve$,
\begin{equation}\label{6.5}
|\Psi_m(\ve)-\Psi_m(\ve_n)|\leq LTe^{LT}(n-1)^{-1}.
\end{equation}
This together with (\ref{6.3}) yields that for all $\ve>0$,
\begin{equation}\label{6.6}
\sup_{0\leq t\leq T}|X^\ve_x(t)-\bar X_x(t)|=O(\ve^{\frac 12-\la})\quad\mbox{a.s.}
\end{equation}
Recalling (\ref{2.4}) and the definition of $R_2^\ve$ we obtain from (\ref{6.7}) that for all $\ve$,
\begin{equation}\label{6.7}
\sup_{0\leq t\leq T}|R_2^\ve(t)|=O(\ve^{1-2\la})\quad\mbox{a.s.}
\end{equation}

Next, by (\ref{3.14}), similarly to the above, we obtain that for $\ve_n=\frac 1n,\, n=1,2,...$
and $\la>0$,
\[
\sup_{0\leq t\leq T}|R_1^{\ve_n}(t)|=O(\ve_n^{1-2\la})\quad\mbox{a.s.}
\]
In order to derive that as $\ve\downarrow 0$,
\begin{equation}\label{6.8}
\sup_{0\leq t\leq T}|R_1^{\ve}(t)|=O(\ve^{1-2\la})\quad\mbox{a.s.}
\end{equation}
we will show that for all $t\in[0,T]$ and all $n>1$,
\begin{equation}\label{6.9}
|R^\ve_1(t)-R^{\ve_n}_1(t)|\leq C_{11}(n-1)^{-1}
\end{equation}
for some $C_{11}>0$ which does not depend on $n$ and $\ve$. Indeed, by (\ref{2.3}) for any $k\leq T/\ve$,
\begin{eqnarray*}
&|\nabla B(\bar X_x(k\ve),\xi(k))-\nabla\bar B(\bar X_x(k\ve))-\nabla B(\bar X_x(k\ve_n),\xi(k))-\nabla\bar B(\bar X_x(k\ve_n))|\\
&\leq 2L^2T(n-1)^{-1}
\end{eqnarray*}
and by (\ref{6.5}),
\[
|X_x^\ve(k\ve)-\bar X_x(k\ve)-X^\ve_x(k\ve_n)+\bar X_x(k\ve_n)|\leq (e^{LT}+1)LT(n-1)^{-1}
\]
which together with (\ref{2.3}) and the definition of $R_1^\ve$ in Lemma \ref{lem3.6} yields (\ref{6.9}), and so also (\ref{6.8}).
Now, (\ref{6.3}), (\ref{6.7}) and (\ref{6.8}) yield (\ref{6.1}) and complete the proof of the lemma.
\end{proof}

Extend the definition of $Z^\ve$ to all $t\in[0,T]$ by setting $Z^\ve(t)=Z^\ve(n\ve)$ if $n\ve\leq t<(n+1)\ve$.
The following estimate leads us to the study of a.s. approximations of $\sqrt\ve S^\ve$ with $S^\ve$ defined in
 Lemma \ref{lem3.6}.
\begin{lemma}\label{lem6.2}
For all $\ve>0$,
\begin{equation}\label{6.10}
\sup_{0\leq t\leq T}|Z^\ve(t)-\sqrt\ve G(t)|\leq\sqrt\ve e^{LT}\big (\sup_{0\leq t\leq T}|\sqrt\ve S^\ve(t)-
\int_0^t\sig(\bar X_x(s)dW(s)| +2\sqrt\ve L^2T^2e^{LT}\big )
\end{equation}
where $G$ is given by (\ref{1.4}).
\end{lemma}
\begin{proof}
First, observe that by (\ref{2.3}),
\[
\sup_{0\leq t\leq T}|S^\ve(t)|\leq 2LT\ve^{-1}.
\]
Hence, by (\ref{2.3}) and (\ref{3.20}),
\[
|Z^\ve(n\ve)|\leq 2LT+\ve L\sum_{k=0}^{n-1}|Z^\ve(k\ve)|,
\]
and so by the discrete time Gronwall inequality
\[
\sup_{0\leq t\leq T}|Z^\ve(t)|\leq 2LTe^{LT}.
\]
This together with (\ref{1.4}), (\ref{2.3}) and (\ref{3.20}) yields
\begin{eqnarray*}
&|Z^\ve(t)-\sqrt\ve G(t)|\leq\sqrt\ve|\sqrt\ve S^\ve(t)-\int_0^t\sig(\bar X_x(s))dW(s)|\\
&+\int_0^t|\nabla\bar B(\bar X_x([s/\ve]\ve))-\nabla\bar B(\bar X_x(s))||Z^\ve(s)|ds\\
&+\int_0^t|\nabla\bar B(\bar X_x(s))||Z^\ve(s)-\sqrt\ve G(s)|ds\\
&\leq\sqrt\ve|\sqrt\ve S^\ve(t)-\int_0^t\sig(\bar X_x(s))dW(s)| +2\ve L^2T^2e^{LT}
+L\int_0^t|Z^\ve(s)-\sqrt\ve G(s)|ds.
\end{eqnarray*}
Applying the Gronwall inequality we arrive at (\ref{6.10}).
\end{proof}

\subsection{A.s. approximations}\label{subsec6.2}

We will assume now that $B(x,\xi)=\Sig(x)\xi$ where $\Sig$ is the matrix function satisfying (\ref{2.23}).
This assumption will enable us to rely on the a.s. approximation result for the sequence $\xi(k),\, k\in\bbZ$
itself which is essentially well known (see, for instance, Theorem 2.1 in \cite{Ki24+} or Theorem 2.2 in \cite{FK} 
and references there) and it says the following.
\begin{proposition}\label{prop6.3} The sequence of random vectors $\xi(n),\,-\infty<n<\infty$ can be redefined
preserving its distributions on a sufficiently rich probability space which contains also a $d$-dimensional Brownian
 motion $\cW$ with the covariance matrix $\vs$ so that for any $N\geq 1$,
 \begin{equation}\label{6.11}
 \sup_{0\leq n\leq N}|\cS(n)-\cW(n)|=O(N^{\frac 12-\del})\quad\mbox{a.s.}
 \end{equation}
 for some $\del>0$ where $\cS(n)=\sum_{0\leq k<n}(\xi(k)-\bar\xi)$ and $\bar\xi=E\xi(0)$.
 \end{proposition}

 Set also $\cS(t)=\cS(n)$ if $n\leq t<n+1$. Observe that in our circumstances $\hat B(x,\xi(k))=\Sig(x)(\xi(k)-\bar\xi)$
 and $S^\ve(t)=\sum_{0\leq k<[t/\ve]}\Sig(\bar X_x(k\ve))(\xi(k)-\bar\xi)$. Next, we set
 \[
 V^\ve(t)=\sum_{0\leq l\leq\ve^{-\gam}t}\Sig(\bar X_x(\ve^\gam l))\sum_{l\ve^{-(1-\gam)}\leq k<(l+1)\ve^{-(1-\gam)}}
 (\xi(k)-\bar\xi),
 \]
 with $\gam\in(0,1)$ to be chosen later on, and estimate
 \begin{equation}\label{6.12}
 |S^\ve(t)-V^\ve(t)|\leq I^\ve(t)
 \end{equation}
 where
 \[
 I^\ve(t)=\sum_{0\leq l\leq\ve^{-\gam}t}\sum_{l\ve^{-(1-\gam)}\leq k<(l+1)\ve^{-(1-\gam)}}(\Sig(\bar X_x(\ve k))-\Sig(\bar X_x(\ve^\gam l)))
 (\xi(k)-\bar\xi).
 \]
 Since
 \begin{equation}\label{6.13}
 |\Sig(\bar X_x(\ve k))-\Sig(\bar X_x(\ve^\gam l))|\leq L^2\ve^\gamma
 \end{equation}
 and the constant $C_1(M)$ in the estimate (\ref{3.4}) of Lemma \ref{lem3.4} depends only on the $L^\infty$ bound of summands
 themselves, we apply (\ref{3.4}) with
 \[
 \eta(k)=\ve^{-\gam}(\Sig(\bar X_x(\ve k))-\Sig(\bar X_x(\ve^\gam l)))(\xi(k)-\bar\xi)
 \]
 to obtain that
 \begin{equation}\label{6.14}
 E\sup_{0\leq t\leq T}|I^\ve(t)|^{2M}\leq C_{12}(M)\ve^{M(2\gam-1)}
 \end{equation}
 for some $C_{12}(M)>0$ which does not depend on $\ve$.

 Next, we set
 \[
 \Xi^\ve(t)=\sum_{0\leq l\leq\ve^{-\gam}t}\Sig(\bar X_x(\ve^\gam l))(\cW((l+1)\ve^{-(1-\gam)})-\cW(l\ve^{-(1-\gam)}))
 \]
 and
 \begin{eqnarray*}
 &J^\ve(t)=\int_0^{t/\ve}\Sig(\bar X_x(\ve s))d\cW(s)-\Xi^\ve(t)\\
 &=\sum_{0\leq l\leq\ve^{-\gam}t}\int_{l\ve^{-(1-\gam)}}^{(l+1)\ve^{-(1-\gam)}}
 (\Sig(\bar X_x(\ve s))-\Sig(\bar X_x(\ve^\gam l)))d\cW(s).
 \end{eqnarray*}
 By the standard martingale estimates of stochastic integrals (see, for instance, \cite{Mao}) we obtain that
  \begin{equation}\label{6.15}
 E\sup_{0\leq t\leq T}|J^\ve(t)|^{2M}\leq C_{13}(M)\ve^{M(2\gam-1)}
 \end{equation}
  for some $C_{13}(M)>0$ which does not depend on $\ve$. Now, we employ (\ref{6.1}) to estimate
  \begin{equation}\label{6.16}
  \sup_{0\leq t\leq T}|V^\ve(t)-\Xi^\ve(t)|=O(\ve^{-\frac 12+\del-\gam})\quad\mbox{a.s.}
  \end{equation}

  Set $\ve_n=\frac 1n,\, n=1,2,....$ Then by (\ref{6.14}),
  \begin{equation}\label{6.17}
  P\{\sqrt {\ve_n}\sup_{0\leq t\leq T}|I^{\ve_n}(t)|>\ve_n^{\gam/2}\}\leq C_{12}(M)\ve_n^{M\gam}.
  \end{equation}
  Taking here $M\geq 2\gam^{-1}$ we obtain by the Borel--Cantelli lemma that
  \begin{equation}\label{6.18}
  \sqrt {\ve_n}\sup_{0\leq t\leq T}|I^{\ve_n}(t)|=O(\ve_n^{\gam/2})\quad\mbox{a.s.}
  \end{equation}
  Now, let $\ve_{n+1}\leq\ve\leq\ve_n$. By (\ref{2.23}),
  \begin{eqnarray}\label{6.19}
& |\sqrt\ve S^\ve(t)-\sqrt\ve_n S^{\ve_n}(t)|\leq 2LT\ve^{-1}(\sqrt {\ve_n}-\sqrt\ve)+2L^2T^2\ve^{-3/2}(\ve_n-\ve)\\
 &\leq 2LT(1+LT)\sqrt {\ve_n}.\nonumber
 \end{eqnarray}
 Again by (\ref{2.3}),
 \[
 |\sqrt\ve V^\ve(t)-\sqrt\ve_n V^{\ve_n}(t)|\leq 4LT\ve^{-\gam}\sqrt {\ve_n}+2L^3T\ve^{-1}(\ve_n^\gam-\ve^\gam)\sqrt {\ve_n}
 \leq C_{14}\ve_n^{\frac 12-\gam}
 \]
 for some $C_{14}>0$ which does not depend on $\ve$. This together with (\ref{6.18}) and (\ref{6.19}) yields that for all $\ve>0$,
 \begin{equation}\label{6.20}
  \sqrt {\ve}\sup_{0\leq t\leq T}|I^\ve(t)|=O(\ve^{\gam/2})\quad\mbox{a.s.}
  \end{equation}
 provided $\gam\leq 1/3$.

Next, by (\ref{6.15}) similarly to (\ref{6.17}),
\[
P\{\sqrt\ve_n\sup_{0\leq t\leq T}|J^{\ve_n}(t)|>\ve^{\gam/2}_n\}\leq C_{13}(M)\ve_n^{M\gam}
\]
and by the Borel--Cantelli lemma
\begin{equation}\label{6.21}
\sqrt\ve_n\sup_{0\leq t\leq T}|J^{\ve_n}(t)|=O(\ve^{\gam/2}_n)\,\,\,\mbox{a.s.}
\end{equation}
Observe that integrating by parts we have
\begin{eqnarray}\label{6.22}
&\int_u^v\big(\Sig(\bar X_x(\ve s))-\Sig(\bar X_x(\ve u))\big)d\cW(s)=\big(\Sig(\bar X_x(\ve v))-\Sig(\bar X_x(\ve u))\big)\cW(v)\\
&-\ve\int_u^v\nabla\Sig(\bar X_x(\ve s))\bar B(\bar X_x(\ve s))\cW(s)ds\nonumber
\end{eqnarray}
where (similarly to (\ref{2.26})) for a matrix function $\Sig(y)=(\Sig_{ij}(y))$ and a vector $\eta=(\eta_1,...,\eta_d)$
we denote by $\nabla\Sig(y)\eta$ the matrix function such that $(\nabla\Sig(y)\eta)_{ij}=\sum_{1\leq k\leq d}\frac {\partial\Sig_{ij}(y)}{\partial y_k}\eta_k$. This together with (\ref{2.23}) yields that
\begin{eqnarray*}
&|J^{\ve_n}(t)-J^\ve(t)|\\
&\leq\sup_{0\leq s\leq T/\ve}|\cW(s)|\sum_{0\leq l\leq\ve^{-\gam}_{n+1}T}\big(|\Sig(\bar X_x(\ve^\gam_n(l+1)))
-\Sig(\bar X_x(\ve^\gam(l+1)))|\\
&+|\Sig(\bar X_x(\ve^\gam_nl))-\Sig(\bar X_x(\ve^\gam l)))|+\ve\int_{l\ve^{-(1-\gam)}}^{(l+1)\ve_n^{-(1-\gam)}}|\nabla\Sig
(\bar X_x(\ve_ns))\bar B(\bar X_x(\ve_ns))\\
&-\nabla\Sig(\bar X_x(\ve s))\bar B(\bar X_x(\ve s))|ds+\ve\int_{l\ve_n^{-(1-\gam)}}^{l\ve^{-(1-\gam)}}|\nabla\Sig
(\bar X_x(\ve_ns))\bar B(\bar X_x(\ve_ns))|ds\\
&+\ve\int_{(l+1)\ve_n^{-(1-\gam)}}^{(l+1)\ve^{-(1-\gam)}}|\nabla\Sig(\bar X_x(\ve s))\bar B(\bar X_x(\ve s))|ds\big)\\
&+4LT\max_{0\leq l\leq\ve^{-\gam} T}|\cW(l\ve_n^{-(1-\gam)})-\cW(l\ve^{-(1-\gam)})|\\
&\leq C_{14}(T)\big(\sup_{0\leq s\leq T/\ve_{n+1}}|\cW(s)|(\ve^{1-\gam}+\ve^\gam)\\
&+\max_{0\leq l\leq\ve_{n+1}^{-\gam}T}|\cW(l\ve_n^{-(1-\gam)})-\cW(l\ve^{-(1-\gam)})|\big)
\end{eqnarray*}
and
\[
|J^\ve(t)|\leq C_{14}(T)\sup_{0\leq s\leq T/\ve}|\cW(s)|
\]
where $C_{14}(T)>0$ does not depend on $\ve$. It follows that
\begin{eqnarray*}
&\cJ_n(T)=\sup_{\ve_{n+1}\leq\ve\leq\ve_n}\sup_{0\leq t\leq T}|J^{\ve_n}(t)-J^\ve(t)|\\
&+(\sqrt\ve_n-\sqrt\ve_{n+1})\sup_{\ve_{n+1}\leq\ve\leq\ve_n}\sup_{0\leq t\leq T}|J^\ve(t)|\\
&\leq C_{14}(T)((n^{-(\frac 32-\gam)}+n^{-(\frac 12+\gam)}+n^{-3/2})\sup_{0\leq t\leq (n+1)T}|\cW(t)|\\
&+n^{-1/2}\max_{0\leq l\leq(n+1)^{\gam}T}\sup_{\ve_{n+1}\leq\ve\leq\ve_n}|\cW(l\ve_n^{-(1-\gam)})-\cW(l\ve^{-(1-\gam)})| .
\end{eqnarray*}

By the standard martingale uniform estimates for the Brownian motion
\[
E\sup_{0\leq t\leq (n+1)T}|\cW(t)|^{2M}\leq\big(\frac {2M}{2M-1}\big)^{2M}T^M(n+2)^M\prod_{k=1}^M(2k-1)
\]
and
\begin{eqnarray*}
&E\max_{0\leq l\leq(n+1)^{\gam}T}\sup_{\ve_{n+1}\leq\ve\leq\ve_n}|\cW(l\ve_n^{-(1-\gam)})-\cW(l\ve^{-(1-\gam)})|\\
& \leq\sum_{0\leq l\leq(n+1)^{\gam}T}E\sup_{0\leq\ve\leq(n(n+1))^{-1}}|\cW(l(n^{-1}-\ve)^{-(1-\gam)}-\cW(ln^{1-\gam})|^{2M}\\
&\leq\big(\frac {2M}{2M-1}\big)^{2M}\sum_{0\leq l\leq(n+1)^{\gam}T}E|\cW(l((n+1)^{1-\gam}-n^{1-\gam}))|^{2M}\\
&\leq 2^{M\gam}\big(\frac {2M}{2M-1}\big)^{2M}(T+1)^{M+1}(n+1)^\gam\prod_{k=1}^M(2k-1).
\end{eqnarray*}
It follows that
\[
E(\cJ_n(T))^{2M}\leq C_{15}(T)n^{-2M\gam}
\]
provided that $\gam<\frac 12$. By the Chebyshev inequality and the Borel--Cantelli lemma we obtain similarly to (\ref{6.21}) that
\[
\cJ_n(T)=O(n^{-\gam/2})=O(\ve_n^{\gam/2})\quad\mbox{a.s.}
\]
This together with (\ref{6.21}) yields that for all $\ve>0$,
\begin{equation}\label{6.23}
\sqrt\ve\sup_{0\leq t\leq T}|J^\ve(t)|=O(\ve^{\gam/2})\quad\mbox{a.s.}
\end{equation}

Introduce $\hat W(t)=\hat W_\ve(t)=\sqrt\ve\cW(t/\ve)$ which is another Brownian motion with the covariance matrix $\vs$ at the
time 1. Then we can choose a standard $d$-dimensional Brownian motion $W=W_\ve$ such that $\vs^{1/2}W_\ve=\hat W_\ve$. Then
\begin{equation}\label{6.24}
\sqrt\ve\int_0^{t/\ve}\Sig(\bar X_x(\ve s))d\cW(s)=\int_0^t\Sig(\bar X_x(u))d\hat W_\ve(u)=\int_0^t\sig(\bar X_x(u))dW_\ve(u).
\end{equation}
Finally, combining (\ref{6.10}), (\ref{6.12}), (\ref{6.16}), (\ref{6.20}), (\ref{6.23}), (\ref{6.24}) and taking $\gam=\del/2$ we
obtain that
\[
\sup_{0\leq t\leq T}|Z^\ve(t)-\sqrt\ve G(t)|=O(\ve^{\frac 12(1+\frac \del 2)})\quad\mbox{a.s.}
\]
provided that the Gaussian process $G$ is given by (\ref{1.4}) with $W=W_\ve$ the same as in (\ref{6.24}).
This together with (\ref{6.3}), (\ref{6.7}) and (\ref{6.8}) completes the proof of Theorem \ref{thm2.6} for the discrete
time case while the corresponding assertion in the continuous time case follows as well, in view of the discretization estimates
of Section \ref{sec5}.
\qed

\begin{remark}\label{rem6.4}
 The assumption that $B(x,\xi)=\Sig(x)\xi$ enables us to employ the strong approximation theorem to the sums of the sequence
 $\xi(n),\,-\infty<n<\infty$ itself. Without this assumption we would have to apply the strong approximation theorem to the sums
 $S^\ve(t)=\sum_{0\leq k<[t/\ve]}(B(\bar X_x(k\ve),\xi(k))-\bar B(\bar X_x(k\ve)))$ whose summands change with $\ve$ and the number of
  summands depends on $\ve$, as well, i.e. we have to deal here with arrays. This was possible
 in Section \ref{sec4} and \ref{sec5} when our goal was to obtain estimates in Theorems \ref{thm2.1} and \ref{thm2.4} for each fixed
 $\ve$ but it is not clear how to apply this machinery to sums with changing summands with the goal to obtain almost sure estimates
 of the form $O(\ve^\del)$ as $\ve\to 0$ and to use this in order to obtain the law of iterated logarithm of Theorem \ref{thm2.7}.
\end{remark}

\subsection{Law of iterated logarithm}\label{subsec6.3}

Let $\bfW$ be the standard $d$-dimensional Brownian motion such that $\vs^{1/2}\bfW=\cW$. Then by (\ref{1.4}) and (\ref{6.24}) the Gaussian
process $G=G_\ve$ satifies
\begin{equation}\label{6.25}
G_\ve(t)=\int_0^t\nabla\bar B(\bar X_x(s))G_\ve(s)ds+\sqrt\ve\int_0^{t/\ve}\sig(\bar X_x(\ve s))d\bfW(s).
\end{equation}
Hence,
\begin{equation*}
G_\ve=\Phi(\vf_\ve)\quad\mbox{where}\quad\vf_\ve(t)=\sqrt\ve\int_0^{t/\ve}\sig(\bar X_x(\ve s))d\bfW(s)
\end{equation*}
with the map $\Phi$ defined by (\ref{2.25}). Integrating by parts we obtain that
\begin{eqnarray*}
&\int_0^{t/\ve}\sig(\bar X_x\ve s))d\bfW(s)=\sig(\bar X_x(t))\bfW(t/\ve)-\ve\int_0^{t/\ve}\nabla\sig(\bar X_x(\ve s))\bfW(s)ds\\
&=\sig(\bar X_x(t))\bfW(t/\ve)-\int_0^t\nabla\sig(\bar X_x(u))\bfW(u/\ve)du.
\end{eqnarray*}
Hence,
\begin{equation}\label{6.26}
G_\ve=\Phi\Psi(\sqrt\ve\zeta_\ve)\quad\mbox{where}\quad\zeta_\ve(t)=\bfW(t/\ve),\,\, t\in[0,T].
\end{equation}

Observe that (\ref{2.25}) is a particular case of the second order Volterra equation, and so it has a unique solution which can be
seen here directly since if $\Phi_1(\vf)$ and $\Phi_2(\vf)$ are two solutions of (\ref{2.25}), then
\[
|\Phi_1(\vf)(t)-\Phi_2(\vf)(t)|\leq L\int_0^t|\Phi_1(\vf)(s)-\Phi_2(\vf)(s)|ds
\]
for all $t\in[0,T]$ which by the Gronwall inequality implies that $\Phi_1(\vf)\equiv\Phi_2(\vf)$. It follows, in particular, that
$\Phi:\,\cC_d[0,T]\to\cC_d[0,T]$ is a linear map. In addition, by (\ref{2.25}) for any two $\vf,\psi\in\cC_d[0,T]$,
\[
|\Phi(\vf)(t)-\Phi(\psi)(t)|\leq L\int_0^t|\Phi(\vf)(s)-\Phi(\psi)(s)|ds+|\vf(t)-\psi(t)|,
\]
and so by the Gronwall inequality,
\[
\|\Phi(\vf)-\Phi(\psi)\|_{[0,T]}\leq e^{LT}\|\vf-\psi\|_{[0,T]},
\]
i.e. $\Phi$ is a continuous map. Concerning $\Psi$ it is clear from (\ref{2.26}) that it is a linear continuous map of
$\cC_d[0,T]$ into itself.

By the Strassen theorem (see \cite{Str}) with probability one as $\ve\to 0$ the set of limit poits in $\cC_d[0,T]$ of
\[
(\frac 2\ve\log\log\frac 1\ve)^{-1/2}\zeta_\ve=\frac {\sqrt\ve\zeta_\ve}{\sqrt{2\log\log\frac 1\ve}}
\]
coincides with the compact set $\cK\subset\cC_d[0,T]$ defined in Section \ref{subsec2.4}. In fact, \cite{Str} deals with the limit points
of $(\frac 2{\ve_n}\log\log\frac 1{\ve_n})^{-1/2}\zeta_{\ve_n}$ as $\ve_n=1/n\to 0$ but it is not difficult to extend the above assertion
to all $\ve\to 0$ and though an even more general result can be found in \cite{Bal}, for readers' convenience we will provide the direct
proof here, as well. Let $\ve_n\leq\ve <\ve_{n-1}$ with $n\geq 4$ and assume that $\log$ is the natural logarithm. Then
\begin{equation}\label{6.27}
|(\frac 2{\ve_n}\log\log\frac 1{\ve_n})^{-1/2}\zeta_{\ve_n}(t)-(\frac 2\ve\log\log\frac 1\ve)^{-1/2}\zeta_\ve(t)|\leq I_1(n)+I_2(n)
\end{equation}
where
\[
I_1(n)=\sup_{\ve_n\leq\ve <\ve_{n-1}}|(\frac 2{\ve_n}\log\log\frac 1{\ve_n})^{-1/2}-(\frac 2\ve\log\log\frac 1\ve)^{-1/2}|
\sup_{0\leq t\leq T}|\bfW(nt)|
\]
and
\begin{eqnarray*}
&I_2(n)=3(n-1)^{-1/2}\sup_{0\leq t\leq T}\sup_{\ve_n\leq\ve <\ve_{n-1}}|\zeta_{\ve_n}(t)-\zeta_\ve(t)|\\
&\leq 6(n-1)^{-1/2}\max_{0\leq k\leq n-1}\sup_{kT\leq t\leq (k+1)T}|\bfW(t)-\bfW(kT)|.
\end{eqnarray*}

Observe that when $3\leq (n-1)<u\leq n$, then
\[
|\frac d{du}(2u\log\log u)^{-1/2}|\leq (2u)^{-3/2}(\log\log u)^{-1/2}(1+\log u\log\log u)^{-1}\leq 2(n-1)^{-3/2}
\]
and $|\ve_n^{-1}-\ve^{-1}|\leq 1$. This together with the standard moment estimates for the Brownian motion yields that
\[
E(I_1(n))^{2M}\leq C_{15}(M)n^{-2M}T^M
\]
for some $C_{15}(M)>0$ which does not depend on $n$. By the Chebyshev inequality
\[
P\{ I_1(n)>n^{-1/2}\}\leq C_{15}(M)n^{-M}T^M.
\]
Taking $M=2$ we obtain from the Borel--Cantelli lemma that with probability one $I_1(n)\leq n^{-1/2}$ for all $n$ large enough, and so
\[
\lim_{n\to\infty}I_1(n)=0\quad\mbox{a.s.}
\]
Again, by the standard moment estimates for the Brownian motion
\begin{eqnarray*}
&E(I_2(n))^{2M}\leq 6^{2M}(n-1)^{-M}\sum_{0\leq k\leq n-1}E\sup_{kT\leq t\leq (k+1)T}|\bfW(t)-\bfW(kT)|^{2M}\\
&\leq C_{16}(M)6^{2M}(n-1)^{-M}nT^M,
\end{eqnarray*}
where $C_{16}(M)>0$ which does not depend on $n$, and in the same way as above we conclude that
\[
\lim_{n\to\infty}I_2(n)=0\quad\mbox{a.s.}
\]
This together with (\ref{6.27}) yields that the sets of limit points of $(\frac 2{\ve_n}\log\log\frac 1{\ve_n})^{-1/2}\zeta_{\ve_n}$
as $n\to\infty$ and of $(\frac 2\ve\log\log\frac 1\ve)^{-1/2}\zeta_\ve$ as $\ve\to 0$ both coincide with $\cK$.

Finally, since $\Phi$ and $\Psi$ are linear and continuous we obtain from (\ref{6.26}) that with probability one as $\ve\to 0$ the set of limit points of $(2\log\log\frac 1\ve)^{-1/2}G_\ve$ coincides with the compact set $\Phi\Psi(\cK))$ which together with (\ref{2.24}) yields
the assertion of Theorem \ref{thm2.7}.
\qed



\end{document}